\documentclass{amsart}
\usepackage{mathpreamble}

\newcommand{\s}{\widetilde{s}}
\newcommand{\g}{\widetilde{g}}
\newcommand{\M}{\widetilde{M}}

\begin{document}

\title{Holomorphic functions on a class of complete Hermitian manifolds}

\author{Zeqing Miao}
\address{School of Mathematical Sciences, Xiamen University, Xiamen, Fujian, 361005, China.}
\email{{19020240157493@stu.xmu.edu.cn}}
\author{Hanyu Wu}
\address{School of Mathematical Sciences, Xiamen University, Xiamen, Fujian, 361005, China.}
\email{{19020230157192@stu.xmu.edu.cn}}
\author{Bo Yang}
\thanks{The third-named author is partially supported by the National Natural Science Foundation of China under grant numbers 11801475, 12141101, and 12271451, and by the Natural Science Foundation of Fujian Province of China under grant number 2019J05012.}
\address{School of Mathematical Sciences, Xiamen University, Xiamen, Fujian, 361005, China.}
\email{{boyang@xmu.edu.cn}}

\date{10/13/2026}

\begin{abstract}
In this paper, we prove a gradient estimate for holomorphic functions on complete Hermitian manifolds with nonnegative second Ricci curvature. As an application, we establish several Liouville-type theorems for holomorphic functions and mappings between complete Hermitian manifolds under suitable curvature assumptions. In particular, on a complete Hermitian manifold with nonnegative second Ricci curvature and Chern torsion of linear growth, any bounded holomorphic function must be constant. Moreover, we construct examples to show that both the gradient estimate and this Liouville theorem are sharp in a suitable sense.  

In the second part of the paper, we study holomorphic functions on simply connected complete Hermitian manifolds with vanishing Chern curvature and torsion norm of sublinear growth. Combining the gradient estimate with methods from sub-Riemannian geometry, we obtain a complete characterization of holomorphic functions of polynomial growth on such manifolds.
\end{abstract}


\maketitle

\markleft{On complete Hermitian manifolds with flat Chern connection}
\markright{On complete Hermitian manifolds with flat Chern connection}

\setcounter{tocdepth}{1}
\tableofcontents

\section{Introduction}

\subsection{Statement of the main results} 

The gradient estimates of Yau \cite{Yau1975} and Cheng--Yau \cite{ChengYau1975} for positive harmonic functions on complete Riemannian manifolds with Ricci curvature bounded below are fundamental tools in geometric analysis. In particular, every harmonic function of sublinear growth on a complete connected Riemannian manifold with nonnegative Ricci curvature is constant. In this work, we are interested in the function theory on a class of Hermitian manifolds which are \emph{complete with respect to the Riemannian connection}. Let $(M^n ,g)$ be a Hermitian manifold of complex dimension $n$, equipped with its Chern connection. Let $\{e_i\}_{i=1}^n$ be a local $(1, 0)$ frame of $T^{1, 0}M$ and $\{\varphi^k\}_{k=1}^n$ the dual frame. Then the structure equations define the torsion form $\tau^i$ as a $(2, 0)$-form and the curvature form $\Theta$ as an $\operatorname{End}(T^{1,0}M)$-valued $(1,1)$-form. The corresponding components of the Chern curvature tensor are defined as $R_{i\overline{p}j\overline{k}} \coloneqq \Theta_j^s(e_i, \overline{e}_p)g_{s\overline{k}}$. We recall that the second Ricci curvature is defined as $Ric_{i\overline{j}}^{(2)}=g^{k\overline{l}}R_{k\overline{l}i\overline{j}}$. Now we write the torsion form $\tau^{i}=T^{i}_{jk}\varphi^j \wedge \varphi^k$ and consider the full torsion tensor $T$ on $TM \otimes \mathbb{C}$ obtained by an extension of $\tau^{i}$. We further introduce the norm of $T$ as $|T|^2=2T_{pq}^i \overline{T_{st}^j} g_{i\overline{j}} g^{p \overline{s}} g^{q\overline{t}}$. For any holomorphic function $f$ on $M$, we define its gradient norm as $|\nabla f|^2 \coloneqq e_i(f)\overline{e_j(f)}g^{i\overline{j}}$.

Let $\mathcal{O}(M)$ be the ring of holomorphic functions on a complex manifold $M$. Consider a complete Hermitian manifold $(M, g)$. Let $p$ be a fixed point and $d(\cdot, p)$ the distance function with respect to $g$. Given a real number $\alpha \geq 0$, we say that $f \in \mathcal{O}(M)$ has \emph{polynomial growth of order at most $\alpha$} if there exists some constant $C(\alpha, f)$ so that
\begin{equation}\label{porder_atmost}
|f(q)| \leq C(d(q, p)+1)^{\alpha}, \ \ \ \ \forall q \in M.
\end{equation}
Let $\mathcal{O}_d(M, g)$ denote the complex vector space of holomorphic functions with polynomial growth of order at most $d$. 

Our first result establishes a gradient estimate for holomorphic functions of polynomial growth on a complete Hermitian manifold with nonnegative second Ricci curvature ($Ric^{(2)} \geq 0$ for short), under an additional assumption on the growth of the torsion. As an immediate corollary, we obtain a Liouville-type result as
in \cite{Yau1975, Cheng1980}; see Corollary \ref{gene_coro} for a similar result on holomorphic mappings. 

\begin{theorem}\label{grad_intro}
Let $(M^n, g)$ be a complete Hermitian manifold with $Ric^{(2)}(g) \geq 0$.  Fix a point $p \in M$. Assume that there exist constants $\delta \geq 0$ and $C_1>0$ such that the torsion satisfies $|T|(q) \le C_1(1+d(q, p))^{\delta}$ on $M$. Given any $f \in \mathcal{O}_{\widetilde{\delta}}(M, g)$ with $\widetilde{\delta} \geq 0$ and any $\lambda>2\widetilde{\delta}-1+\delta$, there exists a constant $C_2>0$, depending on $\lambda$, $f$ and the geometry of $M$, such that
\begin{equation}
    |\nabla f|^2(q) \le C_2\bigl(1+d(q,p)\bigr)^{\lambda},
    \qquad q\in M.
    \label{grade_intro}
\end{equation}
\end{theorem}

\begin{corollary}\label{Liouville_intro}
Let $(M,g)$ be a complete Hermitian manifold with $Ric^{(2)}(g) \geq 0$. Assume that its torsion satisfies $|T|(x) \leq C(1+d(x, p))^{\delta}$ for constants $\delta \geq 0$ and $C>0$.   
\begin{enumerate}[label=(\arabic*)]
    \item  If $0 \leq \delta<1$ and $\widetilde{\delta} < \frac{1-\delta}{2}$, then $\mathcal{O}_{\widetilde{\delta}}(M, g)=\{\text{constant}\}$. 
    \item  If $\delta=1$, any bounded holomorphic function on $M$ is constant.
\end{enumerate}
\end{corollary}

Obviously, Cheng--Yau's gradient estimates work on complete K\"ahler manifolds with nonnegative Ricci curvature. However, it might not directly apply to complete Hermitian manifolds with $Ric^{(2)} \geq 0$. According to Remark \ref{rho(z)general}, the Riemannian Ricci curvature of a Hermitian manifold with vanishing Chern curvature need not be bounded from below. Another example is the complex Heisenberg group $(G, g)$ equipped with the left-invariant Hermitian metric; see \eqref{Heisen_def} and \eqref{3dim_metric1}. As $g$ has vanishing Chern curvature, we know that $Ric^{(2)}(g)=0$. In the meantime, $g$ is balanced, and the complex Laplacian coincides with the Riemannian Laplacian up to a normalization by Lemma \ref{Lapcompre}. Consider the holomorphic function $z_3$ which has quadratic growth; see Subsection \ref{3dim_examp}. As the Riemannian Ricci curvature is bounded by a negative constant, Cheng--Yau's gradient estimate implies that $|\nabla z_3|$ has at most quadratic growth. In comparison, Theorem \ref{grad_intro} shows that $|\nabla z_3|$ has growth order at most $\frac{3}{2}+\epsilon$ for any $\epsilon>0$. We note that $|\nabla z_3|$ actually has linear growth in this particular case.

Next we show that part (2) of Corollary \ref{Liouville_intro} is sharp as there might exist bounded holomorphic functions when the torsion tensor has faster-than-linear growth.

\begin{example}\label{sharp_L_intro}
    Let $\mathbb{D}$ denote the unit disk in $\mathbb{C}$. For any $\epsilon>0$, there exist a positive constant $C$, and a complete Hermitian metric $g_{\epsilon}$ on $M_1 \coloneqq \mathbb{C}^2 \times \mathbb{D}$ such that $Ric^{(2)}(g_{\epsilon}) \geq 1$ and the torsion satisfies $|T|(q) \leq C (1+d(q, p))^{1+\epsilon}$ for any $q \in M$.
\end{example}

In fact, Theorem \ref{grad_intro} is almost optimal in the sense that one cannot expect the gradient estimate \eqref{grade_intro} to remain valid when $\lambda$ is replaced by $2\widetilde{\delta}-1+\delta-\eta$ for any fixed $\eta>0$. Therefore, the only remaining question is whether one can take $\lambda$ in \eqref{grade_intro} to be exactly $2\widetilde{\delta}-1+\delta$. To see this, we construct another family of complete Hermitian metrics on $\mathbb{C}^3$ by further modifying Example \ref{sharp_L_intro}.

\begin{example}\label{sharp_G_intro}
    Fix any $\delta \in (0, 1)$. For any given $\epsilon>0$, there exist positive constants $C_1, C_2$ which depend on $\delta$ and $\epsilon$, and a complete Hermitian metric $g$ on $\mathbb{C}^3$ 
    satisfying the following properties: its torsion satisfies $|T|(q) \leq C_1 (1+d(q, p))^{\delta}$ for any $q \in \mathbb{C}^3$, and $Ric^{(2)}(g) > 0$. Moreover, let $(z_1, z_2, w)$ denote the global coordinate functions on $\mathbb{C}^3$. Then $w \in \mathcal{O}_{1+\epsilon}(\mathbb{C}^3, g)$ and $|\nabla w|^2(q) \leq C_2 (1+d(q, p))^{1+\delta}$ for any $q \in \mathbb{C}^3$. Moreover, the exponents $\delta$ and $1+\delta$ in the respective estimates for $|T|$ and $|\nabla w|^2$ cannot be lowered further.
\end{example}

Theorem \ref{grad_intro} follows from a similar gradient estimate on holomorphic mappings between Hermitian manifolds; see Lemma \ref{gene} for a detailed statement. The proof begins with the Chern--Lu formula on holomorphic mappings between Hermitian manifolds; see Lemma \ref{ChernLu_lemma}. Then we argue by contradiction and use the idea of auxiliary functions in the proof of the maximum principle in Omori \cite{Omori} and Cheng--Yau \cite{ChengYau1975}. With the aid of the complex Laplacian comparison due to Chen--Yang \cite{CY1984_1}, we obtain the desired contradiction along a suitable sequence of points related to the auxiliary functions. Unfortunately, due to the nature of our proof, we do not produce a local gradient estimate as in Cheng-Yau \cite{ChengYau1975}.

In the second half of the paper, we focus on complete Hermitian manifolds with vanishing Chern curvature. For brevity, we call such a manifold a \emph{complete noncompact Chern-flat} (Hermitian) manifold. For example, every complex Lie group equipped with a left-invariant Hermitian metric is a complete Chern-flat manifold; for completeness of left-invariant metrics, see, for example, \cite[p.~294]{Milnor1976}. Recall that a complex manifold $M$ is \emph{complex parallelizable} if there exists a global holomorphic tangent frame on $T^{1,0} M$. Obviously, a complex parallelizable manifold admits a Hermitian metric with zero Chern curvature. Wang \cite{Wang1954} proved that a compact complex parallelizable manifold must be the quotient space of a complex Lie group $G$ by some discrete subgroup $\Gamma$. Boothby \cite{Boothby1958} further showed that any Hermitian manifold with zero Chern curvature is covered by a complex parallelizable manifold. Moreover, if such a manifold is compact, it is holomorphically isometrically covered by a complex Lie group with a left-invariant metric. 

In view of Boothby's results \cite{Boothby1958}, it is natural to study uniformization questions for complete noncompact Chern-flat manifolds. For instance, one may ask under what conditions such a manifold is covered by a complex Lie group equipped with a left-invariant Hermitian metric. We observe the following corollary of Theorem \ref{grad_intro}.

\begin{corollary}\label{subline_intro}
    Let $(M,g)$ be a complete Chern-flat Hermitian manifold. Fix a point $p \in M$ and assume that the Chern torsion satisfies $|T|(q) \le C(1+d(q, p))^{\delta}$ for constants $0 \leq \delta<1$ and $C>0$. Then the universal cover of $(M,g)$ is holomorphically isometric to a complex Lie group with a left-invariant metric.
\end{corollary}

We also refer to \cite{LWY2026} for further generalizations of Corollary \ref{subline_intro}. Note that Corollary \ref{subline_intro} is sharp as the same conclusion might fail if the torsion norm has linear growth. For any given holomorphic function $\rho(z)$ on $\mathbb{C}$, let $(z_1, z_2)$ be the standard coordinate of ${\mathbb C}^2$. We consider a Hermitian metric $g$ on ${\mathbb C}^2$
\begin{equation}
   \omega= \sqrt{-1} ( \varphi^1\wedge \overline{\varphi^1} + \varphi^2 \wedge \overline{\varphi^2} ) 
   \label{C2Herm_intro}
\end{equation}
where $\varphi^1=dz_1$, $\varphi^2=e^{\rho(z_1)} dz_2$. In Proposition \ref{2dim_more}, we show that (\ref{C2Herm_intro}) defines a complete Chern-flat Hermitian metric. Its Chern torsion components under the corresponding $(1, 0)$ frame are $T^1_{12}=0$, $T^2_{12}=\frac{1}{2} \rho'(z_1)$. Thus the torsion norm $|T|^2=\frac{1}{4}|\rho'(z_1)|^2$. We observe that $|T|$ has exactly linear growth if $\rho$ is a quadratic polynomial of $z_1$. In this case, \eqref{C2Herm_intro} cannot be a left-invariant metric on a complex Lie group. Note that any simply connected complex Lie group is Stein, by the theorem of Matsushima--Morimoto \cite{Ma1960,MM1960}. Therefore, the universal cover of the Hermitian manifold in Corollary \ref{subline_intro} is Stein and hence admits an abundance of holomorphic functions. In the following we study quantitative aspects of the function theory on such manifolds. 

Let $(G,g)$ be a simply connected complex Lie group equipped with a left-invariant Hermitian metric $g$. Recall that $\mathcal{O}_d(G,g)$ was defined in \eqref{porder_atmost}. We observe that $\mathcal{O}_d(G,g)$ is independent of the choice of $g$, since any two left-invariant Hermitian metrics on $G$ are quasi-isometric. From now on, we write $\mathcal{O}_d(G)=\mathcal{O}_d(G,g)$ for simplicity. Let $\mathfrak{g}$ be the Lie algebra of $G$. Recall that the lower central series of $\mathfrak g$ is defined by
\begin{equation}\label{l_central_def}
\mathfrak g_1 \supset \mathfrak g_2
\supset \mathfrak g_3\supset \cdots,\ \ \text{where}\ \ \mathfrak g_1=\mathfrak g,\ \  \mathfrak g_{k+1}=[\mathfrak g, \mathfrak g_k],\ \ k\geq 1.    
\end{equation}
Since the lower central series of $\mathfrak g$ stabilizes after finitely many steps, there exists $N\geq 1$ such that $\mathfrak g_N=\mathfrak g_{N+1}$. 
$\mathfrak g$ is called nilpotent if $\mathfrak g_N=\{0\}$. In particular, $\mathfrak g$ is $s$-step nilpotent if $\mathfrak{g}_{s+1}=\{0\}$ and $\mathfrak{g}_{s} \neq \{0\}$. With this convention, an abelian Lie algebra is $1$-step nilpotent. In general, the series stabilizes at a nontrivial subalgebra. Based on this observation, we reduce the characterization of $\mathcal{O}_d(G)$ for a simply-connected complex Lie group to the corresponding problem for a complex nilpotent group.

\begin{theorem}\label{poly_space_char_intro}
   Let $(G,g)$ be a simply connected complex Lie group equipped with a left-invariant Hermitian metric. Suppose that the lower central series of its Lie algebra $\mathfrak{g}$ stabilizes at the $s$-th term $\mathfrak{g}_s$. Let $\widehat{H}$ be the connected Lie subgroup of $G$ associated with the Lie algebra $\mathfrak{g}_s$. Then $\widehat{H}$ is closed, and the quotient group $G/\widehat{H}$ is simply connected and carries a natural left-invariant Hermitian metric $\widehat{g}$ induced by $(G,g)$. Moreover, $\mathcal{O}_d(G)$ is isomorphic to $\mathcal{O}_d(G/\widehat{H})$. In particular, there exist no non-constant holomorphic functions with polynomial growth on $(G, g)$ if and only if its Lie algebra $\mathfrak{g}$ is perfect (i.e. $\mathfrak{g}=[\mathfrak{g}, \mathfrak{g}]$).
\end{theorem}

\begin{corollary}\label{nilpotent_alg_intro}
    Let $(G,g)$ be a simply connected complex Lie group with a left-invariant Hermitian metric. Assume $\operatorname{dim} G=n$. Then $\mathfrak{g}$ is nilpotent if and only if there exist holomorphic functions $f_1,\cdots,f_n$ with polynomial growth which give local coordinates near the identity $\mathbf{1}_G \in G$.
\end{corollary}

We proceed to show a quantitative characterization of $\mathcal{O}_d(G)$ for any simply connected complex nilpotent Lie group.

\begin{theorem}\label{func_summary_intro}

Let $G$ be a simply connected complex Lie group with a left-invariant metric $g$. Assume that $\operatorname{dim} G=n$ and its Lie algebra $\mathfrak{g}$ is $s$-step nilpotent. We choose an adapted basis $\{e_1, \ldots, e_n\}$ of its Lie algebra in the sense of \eqref{adapt_decomp}. According to \eqref{adapted_index}, we assume that each index $1 \leq i \leq n$ can be written as $(d_i, \alpha_{i})$ where $1 \leq \alpha_i \leq n_{d_i}$. Here the integers $d_i\,(1\leq i \leq n)$ exhaust all integers from $1$ to $s$. Let $\exp: \mathfrak{g} \to G$ be the exponential map and $(z^1, \ldots, z^n)$ be the corresponding global coordinates with respect to $\{e_1, \cdots, e_n\}$. Then
        \begin{equation}
             \mathcal{O}_d(G)=
        \operatorname{span} \Bigl\{(z^1)^{\alpha_1}(z^2)^{\alpha_2}\cdots(z^n)^{\alpha_n}\ |\ \alpha_i \in \mathbb{N} \cup \{0\}, \ \alpha_1 d_1+\alpha_2 d_2 +\cdots+\alpha_n d_n \le d \Bigr\}.
        \end{equation}
\end{theorem}

To sum up, Corollary \ref{subline_intro} and Theorems \ref{poly_space_char_intro} and \ref{func_summary_intro} provide a complete characterization of holomorphic functions of polynomial growth on complete simply connected Chern-flat Hermitian manifolds with torsion of sublinear growth.

\begin{comment}
Let $(M, g)$ be a complete Chern-flat Hermitian manifold of complex dimension $n$. Fix a point $p \in M$ and assume that the Chern torsion satisfies $|T|(q) \le C(1+d(q,p))^{\delta}$ for constants $0 \leq \delta<1$ and $C>0$. Then, for any $d\geq 0$, we have
    \begin{equation}
        \operatorname{dim}\mathcal{O}_d(M, g) \leq \operatorname{dim}\mathcal{O}_{\lfloor d \rfloor}(\mathbb{C}^n).\label{uppbound_intro}
    \end{equation}
Here, $\lfloor c \rfloor$ denotes the greatest integer less than or equal to $c$ and $\mathcal{O}_{\lfloor d \rfloor}(\mathbb{C}^n)$ denotes the vector space of complex polynomials of $z_1, \cdots, z_n$ of degree $\leq d$. Moreover, the equality in \eqref{uppbound_intro} holds for some $d\geq 1$ if and only if $(M, g)$ is holomorphically isometric to the complex Euclidean space $\mathbb{C}^n$.    
\end{comment}

Note that the existence of a nonconstant holomorphic function of linear growth on a complete simply connected Chern-flat manifold $(M,g)$ does not imply that $(M,g)$ splits off an isometric $\mathbb{C}$-factor. This is in sharp contrast with the splitting theorem for complete K\"ahler manifolds with nonnegative bisectional curvature established by Ni--Tam \cite{NT2003}. Theorem \ref{func_summary_intro} is partly motivated by works on harmonic functions of linear growth on complete manifolds with nonnegative Ricci curvature \cite{LT1989,Li1995,CCM1995} and results on sharp dimension estimates for holomorphic functions of polynomial growth on complete K\"ahler manifolds with nonnegative bisectional curvature \cite{Ni2004, CFYZ, Liu2016, YZ2022, Chu2025}. Besides the upper bounds on $\operatorname{dim}\mathcal{O}_d(G)$, Theorem \ref{func_summary_intro} also shows the existence of holomorphic functions of polynomial growth in a quantitative way. Precisely, let $G$ be a simply connected complex Lie group whose Lie algebra $\mathfrak g$ is $s$-step nilpotent. We consider the natural adapted basis associated with the lower central series; see \eqref{adapt_decomp}. Theorem \ref{func_summary_intro} states that the global exponential coordinates associated with this basis give a complete characterization of $\mathcal{O}_d(G)$. Our proof is inspired by methods from sub-Riemannian geometry. Pansu's asymptotic theorem \cite{Pansu1983} relates such a nilpotent group to its associated Carnot group. On this Carnot group with the corresponding limit metric (a Carnot-Carath\'eodory distance), its function theory can be studied using the classical Cauchy estimate on anisotropic polydisks. We refer to \cite{LeDonne_B} or Section \ref{subsec_nilpot} for the relevant background.

\subsection{Further discussion}

It is possible that the Hermitian manifold in part (2) of Corollary \ref{Liouville_intro} admits no nonconstant bounded plurisubharmonic (PSH) functions. Recall that Ni \cite[Proposition 4.1]{Ni1998} proved that any bounded $C^2$ PSH function $u$ on a complete K\"ahler manifold $(M^n, g)$ with nonnegative Ricci curvature must satisfy $(\sqrt{-1}\partial \overline{\partial} u)^n=0$. We remark that part (2) of Corollary \ref{Liouville_intro} holds true on any complete Chern-flat manifold without any growth assumption on the torsion; see \cite{LWY2026}.

Uniformization problems on complete Chern-flat manifolds with torsion growth beyond the sublinear range are further considered in \cite{LWY2026}. Conversely, given a simply connected complex parallelizable manifold $M$, can we characterize all complete Chern-flat Hermitian metrics on $M$? As illustrated by Propositions \ref{2dim_more} and \ref{Heisen_non_inv}, $\mathbb{C}^n$ admits complete Chern-flat Hermitian metrics whose torsion has arbitrarily large growth order. By contrast, much less is known for other simply connected complex parallelizable surfaces and threefolds. This leads us to formulate the following questions in low dimensions.

\begin{question}\label{ques3_intro}
Does $\operatorname{SL}(2, \mathbb{C})$ admit a complete Chern-flat Hermitian metric whose torsion tensor has nonconstant norm?
\end{question}

\begin{question}[A possible strengthening of Corollary \ref{nilpotent_alg_intro} in dimension $2$]\label{ques4_intro}
Consider a simply connected complete noncompact Chern-flat surface $(M, g)$. Assume that there exist $f_1, f_2 \in \mathcal{O}(M)$ of polynomial growth which provide local coordinates near a fixed point $p \in M$. Is $(M, g)$ holomorphically isometric to the complex Euclidean space $\mathbb{C}^2$?
\end{question}

We also refer to \cite{LWY2026} for further discussion on Question \ref{ques4_intro} and its higher dimensional generalizations. We have focused on complete simply connected Chern-flat Hermitian manifolds. It is natural to consider function theory on quotients of such manifolds by discrete subgroups of holomorphic automorphisms. In Subsection \ref{quotientSL}, we discuss a result due to Berteloot--Oeljeklaus \cite{BO1988} that a certain noncompact quotient of $\operatorname{SL}(2, \mathbb{C})$ is non-K\"ahler, which provides a nice illustration of the interaction between group actions and function theory on complex manifolds. For further results on complex parallelizable manifolds, particularly on quotients of complex Lie groups, we refer to Akhiezer \cite{Akhiezer1990,Akhiezer_B}, Huckleberry \cite{Huckle1990}, Winkelmann \cite{Winkel1998}, and the references therein.

The paper is organized as follows. Section \ref{Sec_2} reviews basic material on Hermitian geometry, Boothby's work on Chern-flat Hermitian manifolds, and Chen--Yang's complex Laplacian comparison theorem on Hermitian manifolds. Section \ref{Sec_3} is devoted to the proofs of Theorem \ref{grad_intro} and Corollary \ref{Liouville_intro}. Section \ref{Sec_4} contains the proofs of Corollary \ref{subline_intro} and Theorems \ref{poly_space_char_intro} and \ref{func_summary_intro}. Finally, Section \ref{Sec_5} discusses relevant examples on Chern-flat manifolds. Throughout the paper, unless otherwise stated, we use the Einstein summation convention.

\vskip 0.2cm
%\begin{comment}
\noindent\textbf{Acknowledgments.}  This paper is a slightly revised version of the arXiv preprint (arXiv:2608.01594v2), entitled ``Holomorphic functions on complete Hermitian manifolds with flat Chern connection.'' The third-named author would like to thank Fangyang Zheng for helpful discussions on Hermitian geometry over the years.
%\end{comment}

\section{Preliminaries on Hermitian geometry}\label{Sec_2}

\subsection{The torsion of the Chern connection}

We consider a Hermitian manifold $(M, J, g)$. Let $\{e_i\}$ be a local $(1, 0)$ frame and $\{\varphi^i\}$ the dual coframe. By a Hermitian connection we mean a connection which is compatible with both $g$ and $J$. The structure equation for a Hermitian connection states that
\begin{align}
D e_i&=\theta_i^j e_j, \ \ \ \ d\varphi^i=\varphi^k \wedge \theta_k^i+\tau^i.  \label{struct1}\\
d\theta_i^j&=\theta_i^p \wedge \theta_p^j+\Theta_i^j.  \label{struct2}
\end{align}
Here $\theta_i^j$ is called a connection $1$-form, $\tau^i$ a frame-valued torsion $2$-form, and $\Theta_i^j$ an $\operatorname{End}(T^{1, 0}M)$-valued curvature $2$-form. Taking exterior derivatives at both (\ref{struct1}) and (\ref{struct2}), we get two Bianchi identities. We state the first one in the following.
\begin{align}
d\tau^i=-\tau^k \wedge \theta_k^i+\varphi^k \wedge \Theta_k^i.  \label{bianchi_1}
\end{align}

As we may decompose $\theta_i^j$ into a $(1, 0)$ part and a $(0, 1)$ part, any Hermitian connection $D$ decomposes into a $(1, 0)$ part $D^{1, 0}$ and a $(0, 1)$ part $D^{0, 1}$. The \emph{Chern connection} $D$ is defined to be the unique Hermitian connection with $D^{0, 1}=\overline{\partial}$. We use the notation $X_{,i}$ for the covariant derivative of $X$ with respect to \emph{the Chern connection} $D$, and $X_{;i}$ for the covariant derivative with respect to \emph{the Levi-Civita connection} $\nabla$. Note that the torsion $\tau^i$ of the Chern connection $D$ is of type $(2, 0)$ and $\Theta_i^j$ is of type $(1, 1)$. If we write $\tau^k=T_{ij}^k \varphi^i \wedge \varphi^j$, then we have $\tau^k (e_i, e_j)=2T_{ij}^k$. According to \cite{Gauduchon1977, Gauduchon1984}, we introduce Gauduchon's $(1, 0)$-form $\eta$ as
\begin{align}
\eta_i=2T_{ki}^k,\ \     \eta=\eta_i \varphi^i,\ \ \ \theta_{L}=-(\eta+\overline{\eta})  \label{1formdef}.
\end{align}
Then we may check that
\begin{equation}
\partial \omega^{n-1}=-\eta \wedge \omega^{n-1}.    \label{eta_prop}   
\end{equation}
Now we may take conjugate in (\ref{struct1}) to define $D \overline{e_i}=\overline{\theta_i^j} \overline{e_j}$. By such an extension of $D$, we may define the full torsion tensor
for any vector $X, Y$ in $T M \otimes \mathbb{C}=T^{1, 0}M \oplus T^{0,1}M$ by
\[
T(X, Y) \coloneqq D_X Y-D_Y X-[X, Y].
\] It follows that
\begin{align}
T(e_i, e_j)=2T_{ij}^k e_k=\tau^k (e_i, e_j)e_k, \ \ \ T(\overline{e}_i, \overline{e}_j)=2\overline{T_{ij}^k} \overline{e}_k,\ \ \ T(e_i, \overline{e}_j)=0.
\label{fulltorsion}
\end{align}
From the structural equation (\ref{struct2}), we define the Chern curvature tensor as
\begin{equation}
    \Theta_j^k(e_i, \overline{e}_p)=R_{i\overline{p}j}^k,\ \ \ \
    R_{i\overline{p}j\overline{l}}=g_{k\overline{l}} R_{i\overline{p}j}^k.  \label{cherncurv1}
\end{equation}
It is well-known (see \cite{TW2015} for example) that the commutation formulae for covariant derivatives hold.
\begin{align}
&[D_k, D_{\overline{l}}]\alpha_j=-R_{k\overline{l} j}^i \alpha_i,\ \ \
[D_k, D_{\overline{l}}]\overline{\alpha_j}=R_{k\overline{l} \,\, \overline{j}}^{\overline{i}} \overline{\alpha_i}. 
\label{cov_commute1}\\
&[D_p, D_{\overline{q}}] T_{ij}^k=-R_{p \overline{q} i}^l T_{lj}^k-R_{p \overline{q} j}^l T_{il}^k+R_{p \overline{q}l}^k T_{ij}^l.  \label{cov_commute2}
\end{align}
From \eqref{bianchi_1}, we have the following identity; see, for example, \cite[Lemma 7, p.~1209]{YZ2018}.
\begin{lemma}\label{torsiond}
Given a Hermitian manifold $(M, g)$ and a local unitary frame $\{e_1, \cdots, e_n\}$, we have 
\[
2T_{ij, \overline{l}}^k=R_{j \overline{l}  i \overline{k}}-R_{i \overline{l}  j \overline{k}}.
\]
\end{lemma}

Using \eqref{cherncurv1}, we define the Chern holomorphic sectional and bisectional curvatures for vectors $U,V\in T^{1,0}M$ as follows.
\begin{equation}
    BI(U, V)=\frac{R(U, \overline{U}, V, \overline{V})}{g(U, \overline{U})g(V, \overline{V})}=\frac{R_{i\overline{j}k\overline{l}} U^i \overline{U^j} V^k \overline{V^l}} {g(U, \overline{U})g(V, \overline{V})},\ \ \  H(U) =\frac{R(U, \overline{U}, U, \overline{U})}{g(U, \overline{U})^2}.   \label{cherncurv2}
\end{equation}
Let $c$ be a constant real number. We say a Hermitian manifold $(M, g)$ satisfies $BI(g) \geq c$ if $BI(U, V) \geq c$ for any $U, V \in T^{1, 0} M$, and $H(g) \geq c$ if $H(U) \geq c$ for any $U \in T^{1, 0} M$. We further define three types of Ricci curvature as follows.
\begin{equation}
    Ric_{i\overline{j}}^{(1)}=g^{k\overline{l}}R_{i\overline{j}k\overline{l}},\ \ Ric_{i\overline{j}}^{(2)}=g^{k\overline{l}}R_{k\overline{l}i\overline{j}},\ \ Ric_{i\overline{j}}^{(3)}=g^{k\overline{l}}R_{i\overline{l}k\overline{j}}.  \label{riccicurv}
\end{equation}

\subsection{The complex Laplacian on a Hermitian manifold}

On a Hermitian manifold $(M, g)$, we choose $\{e_i\}_{i=1}^n$ as a local \emph{unitary} frame for $T^{1, 0} M$ and $\{\varphi^i\}$ its dual unitary coframe. Given a real-valued function $u \in C^{\infty}(M)$, we write $u_k=e_k(u)$ for simplicity. We define the \emph{complex Laplacian}
\begin{align}
\Box_g u \coloneqq u,_{i\overline{i}} \coloneqq D^{2}u(e_i, \overline{e_i}) \coloneqq D_{\overline{e_i}} Du(e_i)=\overline{e_i} e_i  (u)-\theta_i^k (\overline{e_i}) e_k(u).    \label{boxdef}
\end{align}
It turns out that $\Box_g u$ is also the trace of $\sqrt{-1}\partial \overline{\partial} u$. Indeed, from the structural equation (\ref{struct1}) and the fact that the torsion is $(2, 0)$, we have the following. 
\begin{align*}
\partial \overline{\partial} u=e_i\overline{e_j}(u)\overline{\varphi^j}+\overline{e_j}(d\overline{\varphi^j})^{(1, 1)}=(e_i \overline{e_j} (u)-\overline{\theta_j^k}(e_i) \overline{e_k}(u)) \varphi^i \wedge \overline{\varphi^j}=u,_{\overline{j}i}\varphi^i \wedge \overline{\varphi^j}.
\end{align*}
We also note that $\partial \overline{\partial} u=u,_{\overline{j}i}\varphi^i \wedge \overline{\varphi^j}=u,_{i\overline{j}}\varphi^i \wedge \overline{\varphi^j}$. To see it, we calculate 
\begin{align*}
u,_{\overline{j}i}=e_i \overline{e_j} (u)-\overline{\theta_j^k}(e_i) \overline{e_k}(u),\ \ \ \
u,_{i\overline{j}}=\overline{e_j} e_i  (u)-\theta_i^k (\overline{e_j}) e_k(u),
\end{align*} Therefore, $u,_{i\overline{j}}-u,_{\overline{j}i}=du(T(e_i, \overline{e_j}))=0$. 

Next we recall a standard lemma on the relation between the complex Laplacian $\Box_g$ with the (Riemannian) Laplace-Beltrami operator $\Delta_g$. It follows from a formula which relates the connection forms of the Chern connection and the Levi-Civita connection; see, for example, \cite[Lemma~2, p.~1203]{YZ2018}.

\begin{lemma}\label{Lapcompre}
Let $f \in C^{\infty}(M)$ be real-valued. Then
\begin{align}
\Box_g f=\frac{1}{2}\Delta_g f+\operatorname{Re}(\eta_p \overline{f_p}). \label{twolap}
\end{align}
\end{lemma}

\subsection{Boothby's result on Hermitian manifolds with zero Chern curvature}

In this subsection, we review Boothby's result on compact Hermitian manifolds with zero Chern curvature in \cite{Boothby1958}. We begin with a fundamental result on a flat Chern connection.

\begin{lemma}[{Boothby \cite[Theorems 1, 2, and 3]{Boothby1958}}]\label{flat_para}
Any Hermitian manifold $(M, g)$ with zero Chern curvature admits a local holomorphic unitary frame which is parallel with respect to the Chern connection. In particular, such a frame can be made global if $(M, g)$ is simply connected. Conversely, any complex parallelizable (by a holomorphic frame $\{e\}$) complex manifold $M$ admits a Hermitian metric $g$ with flat Chern curvature so that $\{e\}$ is unitary and parallel. 
\end{lemma}

Consider any Hermitian manifold $(M, g)$ with zero Chern curvature. We pick a local unitary frame 
$\{e_i\}$ and the dual frame $\{\varphi^i\}$, so that the corresponding Hermitian metric $g$ can be expressed as
\[
\omega=\sqrt{-1}(\sum_{i=1}^n \varphi^i \wedge \overline{\varphi}^i).
\]
Now we consider the norm of Chern torsion tensor $\Psi=\sum |T_{jk}^i|^2$. We note that $T_{ij, \overline{l}}^k=0$ by Lemma \ref{torsiond}. Together with (\ref{cov_commute2}), we may solve
\[
\Box_g \Psi=\sum |T_{jk, p}^i|^2 \geq 0.
\]
Now we assume that $M$ is \emph{compact}. In view of Lemma \ref{Lapcompre}, the strong maximum principle for second-order elliptic operators applies to $\Box_g$. We conclude that $\Psi$ is constant and $T_{jk, p}^i=0$ on $M$. It follows that $T_{ij}^k$ is locally constant under the frame $\{e_i\}$. Then the structure equation (\ref{struct1}) reduces to $[e_i, e_j]=-2T_{ij}^k e_k$. Moreover, (\ref{bianchi_1}) is reduced to the Jacobi identity 
\begin{equation}\label{jacobi}
\sum_{m} T_{jl}^{m} T_{im}^{k}+T_{li}^{m} T_{jm}^{k}+T_{ij}^{m} T_{lm}^{k}=0.     
\end{equation}
Let $(\widetilde{M},\widetilde{g})$ denote the universal cover of $(M,g)$. The local frame $\{e_i\}$, together with its dual coframe $\{\varphi^i\}$, lifts to a global parallel unitary frame $\{\widetilde{e}_i\}$ and dual coframe $\{\widetilde{\varphi}^i\}$ on $(\widetilde{M},\widetilde{g})$. According to \eqref{jacobi}, $\operatorname{span}_{\mathbb{C}}\{\widetilde{e}_1, \ldots, \widetilde{e}_n\} \subset H^0(\widetilde{M}, T^{1, 0}\widetilde{M})$ is isomorphic to a complex $n$-dimensional Lie algebra $\mathfrak{g}$. It follows from \cite[pp.~188-192]{Cartan1937} that $(\widetilde{M}, \widetilde{g})$ admits a complex Lie group structure such that $\widetilde{g}$ becomes a left-invariant metric.

\begin{remark}\label{Cartan_add}
   Let $G$ be the simply connected complex Lie group with its Lie algebra $\mathfrak{g}$ and $\rho: \mathfrak{g} \rightarrow \operatorname{span}_{\mathbb{C}}\{\widetilde{e}_1, \ldots, \widetilde{e}_n\}$ the isomorphism in the above argument. Let $L_x: G \ni y  \to xy \in G$ denote the left translation by $x \in G$. Let $g_{G}$ be a left-invariant Hermitian metric on $G$ such that $\{(L_x)_{\ast}(E_i)\}$ is a unitary frame for any $x \in G$. Note that $g_G$ is complete, see \cite[p.~294]{Milnor1976} for example. Without using Cartan's result in \cite{Cartan1937}, we can prove that $(\widetilde{M}, \widetilde{g})$ is holomorphically isometric to $(G, g_{G})$ directly. There are two (essentially equivalent) approaches as follows. 
  \begin{enumerate}
      \item  The first method is to construct a holomorphic map from $G$ to $\widetilde{M}$ as an orbit of a holomorphic flow. Namely, choose any $x \in G$. Let \(\gamma:[0,1]\rightarrow G\)
      be a smooth path from the identity $e$ to $x$. Define $\xi_\gamma(t)=\bigl(L_{\gamma(t)^{-1}}\bigr)_{\ast} \gamma^{\prime}(t)$. Then we consider
      \[
      \frac{d}{dt} q_{\gamma}(t)=\rho(\xi_\gamma(t))(q_{\gamma}(t)),\ \  q_{\gamma}(0)=\widetilde{p} \in \widetilde{M}.
      \]
      Since $\rho(\xi_\gamma(t))$ has a bounded length for $t \in [0, 1]$ with respect to $\widetilde{g}$, we may solve $q_{\gamma}(1)$ for a fixed point $\widetilde{p} \in \widetilde{M}$. Moreover, $q_{\gamma}(1)$ is independent of the choice of $\gamma(t)$. So we have a map $F:G \ni x  \rightarrow q_{\gamma}(1) \in \widetilde{M}$. Next we show that $F$ is a holomorphic local isometry. In view of \cite[Theorem 5.4, p.~116]{Sakai}, \(F\) is a holomorphic covering map and, in particular, an isometry. 
      \item  The second one constructs an involutive distribution on the product manifold $(G\times \widetilde{M}, g_{G} \oplus \widetilde{g})$. Now we define a distribution $\mathcal{D} \subset T(G\times \widetilde{M})$ such that
      \[
      \mathcal{D}_{(x, q)}=\{\bigl((L_x)_{\ast}(X),\rho(X)(q)\bigr) \in T_x^{1,0} G \times T_{q}^{1,0} \widetilde{M} \ \ | \ \ X \in \mathfrak{g}\}.
      \]
      Observe that $\mathcal{D}$ is involutive. Therefore there exists a maximal integral manifold $L \subset G \times \widetilde{M}$ through a fixed point $(e, \widetilde{p})$.
      Let $\operatorname{pr}_1: L \to G$ and $\operatorname{pr}_2: L \to \M$ be the projection map to each factor. 
      One can verify that $\operatorname{pr}_1: (L, \frac{1}{2} (g_{G} \oplus \widetilde{g})|_L) \to (G, g_{G})$ is a local isometry. We may check that $(L, \frac{1}{2} (g_{G} \oplus \widetilde{g})|_L)$ is complete, and hence $\operatorname{pr}_1$ is a covering map by \cite[Theorem 5.4 on p.~116]{Sakai} again. It follows that $\operatorname{pr}_1$ is a holomorphic isometry. Similarly, $\operatorname{pr}_2$ is also a holomorphic isometry. Therefore, $F= \operatorname{pr}_2 \circ (\operatorname{{pr}_1})^{-1}: G \to \widetilde{M}$ is the desired isometry. 
\end{enumerate}
\end{remark}

We summarize the above discussion in the following.

\begin{theorem}[Boothby {\cite[Theorem 4]{Boothby1958}}]\label{boothby_main}
The universal cover of any compact Hermitian manifold $(M, g)$ with flat Chern curvature is holomorphically isometric to a complex Lie group $G$ with a left invariant metric.
\end{theorem}

\begin{comment}
On a noncompact Hermitian manifold $(M, g)$ with zero Chern curvature, the Chern torsion may not have a constant norm. Indeed, there are examples whose universal coverings are not isometric to complex Lie groups with left-invariant Hermitian structures. We refer to \cite[Example on p.~231]{Boothby1958} for an example on $\mathbb{C}^2 \setminus \{z_1z_2=0\}$ which is incomplete with respect to the Riemannian connection. Moreover, a complete example on $\mathbb{C}^2$ is constructed in \cite[p.~5767]{WYZ2020}.    
\end{comment}

\subsection{Complex Laplacian comparison on Hermitian manifolds}

\begin{comment}
 \begin{theorem}[Chen--Yang \cite{CY1984_2}]\label{CYschwarz}
    Let $(M, g)$ be a complete Hermitian manifold with a fixed point $p \in M$. Assume that its Chern torsion satisfies
    \[
    \lim_{q \rightarrow \infty} \frac{|T|(q)}{d_g(q, p)}=0.
    \]
    We assume either of the following conditions holds.
    \begin{enumerate}[label=(\arabic*)]
        \item  $H(g) \geq -k_1$ for some constant $k_1 \geq 0$ and  \[
    \lim_{q \rightarrow \infty} \frac{\inf_{U, V \in T_q^{1, 0} M} BI(q; U, V)}{d_g(q, p)^2}=0.
    \]
    \item  $Ric^{(2)}(g) \geq -k_1$ for some constant $k_1 \geq 0$. 
    \end{enumerate}
    Let $(N, h)$ be a Hermitian manifold with holomorphic sectional curvature $H(h) \leq -k_2$ for some constant $k_2>0$.
    Then any holomorphic map $f: M  \rightarrow N$ satisfies $f^{\ast} h \leq \frac{k_1}{k_2} g$.
\end{theorem}   
\end{comment}

We recall a complex Laplacian comparison result on Hermitian manifolds.

\begin{lemma}[Chen--Yang {\cite[Proposition 2, p.~636]{CY1984_1}}]\label{cLap_comp}
    Let $(M^n,g)$ be a complete Hermitian manifold and $d(x)=d_g(p, x)$ be the distance function from $p \in M$. Suppose the norm of its Chern torsion is bounded by $A_1$. Then at any point outside the cut locus of $p$, we have the following result.
    \begin{enumerate}[label=(\arabic*)]
    \item  If $BI(g) \geq -A_2$ for $A_2 \geq 0$, then for a local frame $\{e_i\}$,
    \begin{equation}
       D^{2}d(e_i, \overline{e}_j) \le (\frac{1}{d}+2(A_1+\sqrt{A_2})) g_{i \overline{j}}.     
    \end{equation}
    \item  If $Ric^{(2)}(g) \geq -A_3$ for $A_3 \geq 0$, then
    \begin{equation}
        \Box_g\, d \le \frac{n}{d} +2n(A_1+\sqrt{A_3}).
    \end{equation}
    \end{enumerate}
\end{lemma}

\begin{remark}
     Since the proof of Lemma \ref{cLap_comp} is local, i.e., the same conclusions hold on a geodesic ball, $A_1,A_2$ and $A_3$ can be taken as increasing functions on the radii of geodesic balls. Analogous estimates also hold on almost Hermitian manifolds; see \cite[Theorem 4.2]{Tosatti} and \cite[Theorem 1.1]{Yu2018} for further details.
\end{remark}

\begin{comment}
Consider a complete simply connected Hermitian manifold with zero Chern curvature. By Lemmas \ref{torsiond} and \ref{flat_para}, we choose a global holomorphic unitary frame $\{e\}$ so that each component $T_{ij}^k$ is holomorphic on $M$. As a consequence of Theorem \ref{CYschwarz}. $T_{ij}^k$ is constant if it is bounded. By Lemma \ref{flat_para} and Remark \ref{Cartan_add}, we obtain the following immediate generalization of Theorem \ref{boothby_main}.

\begin{proposition}\label{torsion_bd}
Any simply connected complete Hermitian manifold $(M, g)$ with zero Chern curvature and bounded torsion is holomorphically isometric to a complex Lie group with a left-invariant metric. 
\end{proposition}    
\end{comment}

\section{Liouville theorems and applications}\label{Sec_3}

In this section, we prove Liouville-type theorems for holomorphic functions and holomorphic mappings from complete Hermitian manifolds with nonnegative second Ricci curvature. The main result is Theorem \ref{grad}, whose proof relies on Lemma \ref{gene}; together, they provide a mapping analogue of Theorem \ref{grad_intro}.

\subsection{Basic facts on holomorphic maps between Hermitian manifolds}\label{mapping_term}

Let $f: M \rightarrow N$ be a smooth map between two Hermitian manifolds $(M, g)$ and $(N, h)$. We consider $df$ as a section of $T^{\ast} M \otimes f^{*}TN$. In terms of local coordinates, we write
\[
df=\frac{\partial f^{\alpha}}{\partial x^{i}} dx^i \otimes \frac{\partial}{\partial y^{\alpha}}.
\]
Now we introduce a natural connection $D$ on $T^{\ast} M \otimes f^{\ast}TN$.
\begin{equation}
    D\,df(X, Y)=D_X df(Y) \coloneqq D^{N}_{df(X)} (df(Y))-df(D^{M}_X Y).  \label{ind_connect}
\end{equation} Here $D^M$ and $D^N$ denote the Chern connections on $M$ and $N$ respectively. It follows that $D\,df$ is a section of $T^{\ast}M \otimes T^{\ast} M \otimes f^{\ast}TN$,  hence an analog of the second fundamental form when $f$ is an isometric immersion between Riemannian manifolds. After complexification, $D\,df$ can be viewed as a section of $(T^{\ast}M \otimes \mathbb{C}) \otimes (T^{\ast}M \otimes \mathbb{C}) \otimes f^{\ast}TN$. Then it admits the following decomposition
\[
D\,df=(D\,df)^{(2, 0)}+(D\,df)^{(1, 1)}+(D\,df)^{(0, 2)}.
\]
Jost--Yau \cite{JY1993} introduced the notion of a \textit{Hermitian harmonic map} by requiring
\(
\operatorname{tr}_g\bigl((Ddf)^{(1,1)}\bigr)=0.
\) Motivated by their definition, we call $f$ a \textit{Hermitian totally geodesic map} if $D\,df=0$. We refer to \cite{JY1993} for applications of Hermitian harmonic maps on the rigidity problems of complex manifolds. Note that any holomorphic map between Hermitian manifolds is Hermitian harmonic.

Next, we recall the Chern--Lu inequality (\cite{Chern1968} and \cite{Lu1968}) on holomorphic mappings between Hermitian manifolds. We follow the exposition in \cite[Lemma 7.22, p.~183]{Zheng_b} to state the result in terms of local holomorphic coordinates.

\begin{lemma}[The Chern--Lu inequality]\label{ChernLu_lemma}
    Let $f : M \to N$ be a holomorphic map between two Hermitian manifolds $(M, g)$ and $(N, h)$. Let $R^M$ and $R^N$ be the Chern curvature tensors of $M$ and $N$ respectively. Let $V \in T_p^{1, 0} M$, and $(U, z^i)$ and $(W, w^{\alpha})$ be the corresponding local holomorphic coordinates on $M$ and $N$. We write 
    \[
     f_i^\alpha=\frac{\partial w^{\alpha}(f(z))}{\partial z^i},\ \ f_{ik}^\alpha=\frac{\partial^2 w^{\alpha}(f(z))}{\partial z^i \partial z^k},\ \ 
     \Psi_{ik}^{\alpha}=g^{j\overline{p}}g_{i\overline{p}, k} f_j^{\alpha},\ \text{and}\ \Phi_{ik}^{\alpha}=f_i^{\beta} f_k^{\gamma} h^{\alpha \overline{\eta}} h_{\beta\overline{\eta}, \gamma}.
    \]
    Let $u=\operatorname{tr}_g f^*h=g^{i\overline{j}} h_{\alpha \overline{\beta}} f_i^{\alpha} \overline{f_j^{\beta}}$ be the energy density of $f$. Then we have 
    \begin{align}
    <\partial \overline{\partial} u, V \wedge \overline{V}>=&
    R^M(V, \overline{V}, g^{p\overline{j}} \overline{f_j^\beta} \frac{\partial}{\partial z^p}, \overline{g^{q\overline{i}} \overline{f_i^{\alpha}} \frac{\partial}{\partial z^q}} ) h_{\alpha\overline{\beta}} \label{chernlu1} \\
    &-R^N(\partial f(V), \overline{\partial f(V)}, \partial f(\frac{\partial}{\partial z^i}), \overline{\partial f(\frac{\partial}{\partial z^j})})g^{i\overline{j}} \nonumber\\
    &+\bigl|f_{ik}^{\alpha}V^k-\Psi_{ik}^{\alpha}V^k+\Phi_{ik}^{\alpha}V^k\bigr|^2.  \nonumber
    \end{align}
    Note that the last term on the right-hand side of \eqref{chernlu1} is exactly 
    $|D\,df(V, \cdot)|^2$ where $D\,df$ is defined in \eqref{ind_connect}. As a result, we have
    \begin{align}
    \Box u=&
    Ric_M^{(2)}(g^{p\overline{j}} \overline{f_j^\beta} \frac{\partial}{\partial z^p}, \overline{g^{q\overline{i}} \overline{f_i^{\alpha}} \frac{\partial}{\partial z^q}} ) h_{\alpha\overline{\beta}}
    \label{chernlu2}  \\
    &-R^N(\partial f(\frac{\partial}{\partial z^k}), \overline{\partial f(\frac{\partial}{\partial z^l})}, \partial f(\frac{\partial}{\partial z^i}), \overline{\partial f(\frac{\partial}{\partial z^j})})g^{k\overline{l}} g^{i\overline{j}}+|D\,df(\cdot, \cdot)|^2.  \nonumber
    \end{align}      
\end{lemma}

\subsection{On a class of rotationally symmetric Hermitian metrics on \texorpdfstring{$\mathbb{C}^m$}{C^m}}  \label{sec_rota}

In this subsection, we construct certain rotationally symmetric Hermitian metrics on the complex Euclidean space, with prescribed curvature properties. These metrics will serve as models for our application of the Schwarz lemma; see the proof of Theorem \ref{grad}.

Let $h$ be a complete rotationally symmetric metric on $\mathbb{C}$ with nonpositive Gaussian curvature. We denote by $r$ and $s$ the Euclidean distance and the $h$-geodesic distance from the origin $O$, respectively. Then we may express
\begin{equation}
    h=\xi(r) (dr^2 +r^2 d\theta^2)=ds^2 +(f(s))^2 d\theta^2, \ \text{where}\ \xi(r)=(s'(r))^2,\ \text{and}\ r s'(r) =f(s).\label{rs}
\end{equation}
Note that $f \in C^{\infty}[0, \infty)$ satisfies the Jacobi equation
\begin{equation}
    f''(s) +K(s)f(s)=0, \ \ f(0)=0, \ \ f'(0)=1.  \label{Jaco_eqn}
\end{equation} 
Here, $K(s)$ is the Gauss curvature of $h$. Conversely, for any given $K(s)$ which can be written as $K(s)=\widehat{K}(s^2)$ for some $\widehat{K}(s) \in C^{\infty}[0, \infty)$, we may solve $f(s)$ from \eqref{Jaco_eqn} which satisfies $f^{(2k)}(0)=0$ for any $k \in \mathbb{N}$. Then we may determine the conformal factor $\xi$ by the equation in \eqref{rs} up to any prescribed constant $s_0>0$. That is, for any $s_0>0$, we define
$r=\exp\left(\int_{s_0}^{s}\frac{1}{f(\tau)}\,d\tau\right)$. As $f(s)=s+O(s^3)$ when $s \rightarrow 0+$, we see that 
$r(s)$ well-defined for $s \in [0, \infty)$ with $r(0)=0$ and $r(s_0)=1$. In other words, we have solved $s=s(r)$ with the normalization $s(1)=s_0$. Finally we determine $\xi(r)=(s'(r))^2$.

\begin{lemma}\label{plane}
   For any $\varepsilon>0$, there exist constants $A=A(\varepsilon)>0$ which is sufficiently small, and $C=C(\varepsilon)>0$ such that $h$ is a complete metric on $\mathbb{C}$ of the form \eqref{rs}, satisfying
    \begin{equation}
        K(s)=-\frac{A}{(1+s^2)^{1+\frac{\varepsilon}{2}}}, \ \ \  \textit{and} \ \ \
        s \le C (1+r)^{1+\varepsilon}\ \text{for some}\ C>0. \label{distance1}
    \end{equation}
\end{lemma}
\begin{proof}[Proof of Lemma \ref{plane}]
    
    For any $K(s) \in C^{\infty}[0, \infty)$, the initial value problem \eqref{Jaco_eqn} admits a smooth solution $f$ on $[0, \infty)$. As $K(s)=-\frac{A}{(1+s^2)^{1+\frac{\varepsilon}{2}}}$, we have $\int_0^{+\infty} -sK(s)\,ds < \infty$. It follows from \cite[Lemma 4.5, p.~62]{GW79} that 
    \begin{equation*}
       s \leq f(s) \le \eta s, \ \ \ \textit{where} \ \ \eta =e^{-\int_0^{+\infty} sK(s)\,ds}.
    \end{equation*}
    We choose $A>0$ sufficiently small so that $\eta \leq 1+\varepsilon$. Note that \eqref{rs} implies that $(\log s(r))' \le \frac{\eta}{r}$. Let $s_0=s(1)$, we get $s \leq s_0 r^{\eta} \leq C r^{1+\varepsilon}$ for any $r \geq 1$. As $f(s) \geq s$ and $f'(s) >1$ for $s>0$, we observe that
    $\frac{df}{ds}=1+\frac{d(\frac{ds}{dr})}{ds} r \geq 1$. It follows that $\xi(r)=(\frac{ds}{dr})^2$ is increasing with respect to $r>0$. Then for any $r>0$ we have
    \begin{equation}\label{xi(0)_formula}
    \xi(r) \geq \xi(0)=(\frac{ds}{dr}|_{r=0})^2=(\lim_{r \rightarrow 0} \frac{s}{r})^2.    
    \end{equation}
    Note that $f(s)=s+O(s^3)$ near $s=0$.
    \[
    \lim_{s \rightarrow 0} \frac{s}{r}=\lim_{s \rightarrow 0} \frac{s}{\exp\left(\int_{s_0}^s \left(\frac{1}{f(\tau)}-\frac{1}{\tau}\right)d\tau\right) \frac{s}{s_0}}=\frac{s_0}{\exp\left(\int_0^{s_0} \left(\frac{1}{\tau}-\frac{1}{f(\tau)}\right)d\tau\right)}.
    \]
    Finally, we prove
    \begin{equation}\label{xi(0)_bound}
     \xi(0) \ge s_0^2 e^{-\frac{1}{6}A\eta s_0^2}.
     \end{equation}
    To see it, recall that \eqref{Jaco_eqn} implies
    \(
    f''(s)=\frac{A}{(1+s^2)^{1+\frac{\varepsilon}{2}}} f(s) \le A\eta s.
    \)
    Together with $f(0)=0$ and $f'(0)=1$, we estimate 
    \[
    f(s)-s=\int_0^s (s-\tau)f''(\tau)\,d\tau \le A\eta s \int_0^s (s-\tau)\,d\tau =\frac{1}{6}A\eta s^3.
    \]
    Hence for any $s>0$, 
    \[
    \frac{1}{s}-\frac{1}{f(s)} =\frac{f(s)-s}{sf(s)} \le \frac{1}{6}A\eta s.
    \]
    It follows that
    \[
    \int_0^{s_0} \left(\frac{1}{\tau}-\frac{1}{f(\tau)}\right)d\tau \le \frac{1}{12}A\eta s_0^2,
    \]
    Then \eqref{xi(0)_bound} follows from \eqref{xi(0)_formula}.

\end{proof}

Let $r$ (and $s$) denote the Euclidean (and geodesic) distance from the origin $O$ with respect to a complete Hermitian metric on $\mathbb{C}^m$. Then Lemma \ref{plane} can be generalized to higher dimensions. 

\begin{lemma}\label{planetocm}
Given any integer $m \geq 1$ and any $\varepsilon>0$, there exist two constants $C_1, C_2>0$ and a complete Hermitian metric $h$ on $\mathbb{C}^m$ such that
    \begin{equation}
        BI(h) \leq -\frac{C_2}{(1+s)^{2+\varepsilon}}, \ \ \  \textit{and} \ \ \
        s \le C_1 (1+r)^{1+\varepsilon}. \label{distance2}
    \end{equation} 
\end{lemma}
    
\begin{proof}[Proof of Lemma \ref{planetocm}]
    We consider a rotationally symmetric Hermitian metric 
    \[
    \omega_h=\sqrt{-1}e^{2\varphi} (dz^1 \wedge d\overline{z^1}+\cdots+dz^m \wedge d\overline{z^m})
    \] where $\varphi(r) \in C^{\infty}[0, \infty)$ is to be determined. We recall a general formula for the Chern curvature under conformal changes; see, for example, \cite[p.~1215]{YZ2018}. Under the unitary frame ${e_i=e^{-\varphi}\frac{\partial}{\partial z^i}}$, we write $R_{k\overline{l}i\overline{j}}(h)=R(e_k, \overline{e_l}, e_i, \overline{e_j})$. Then we get
    \begin{equation}
        R_{k\overline{l}i\overline{j}}
        =-2e^{-2\varphi} \frac{\partial^2 \varphi}{\partial z^k \partial \overline{z^l}} \delta_{ij}.   \label{conf_curva1}
    \end{equation}
    Let $g_{\mathbb{S}^{2m-1}}$ denote the round metric with constant sectional curvature $1$ on $\mathbb{S}^{2m-1}$. As 
    \[
     h=e^{2\varphi}(dr^2+r^2g_{\mathbb{S}^{2m-1}})=ds^2+(\rho(s))^2g_{\mathbb{S}^{2m-1}},
    \] we have $s'(r)=e^{\varphi}$ and $\rho=re^\varphi$. Now we compute the Chern curvature of $h$ along a ray $p=(|z|, 0, \cdots, 0)$. Note that $e_1$ is the radial direction in this case. If we write $\tau=r^2$, then we may solve all nonzero curvature components from (\ref{conf_curva1}) as follows.
    \begin{align}
      R_{1\overline{1}1\overline{1}}=R_{1\overline{1}k\overline{k}}=-\frac{2(\tau \varphi^{\prime\prime}(\tau)+\varphi^{\prime}(\tau))}{e^{2\varphi}},\ \ \ \ R_{k\overline{k}p\overline{p}}=R_{k\overline{k}1\overline{1}}=-\frac{2\varphi^{\prime}(\tau)}{e^{2\varphi}},\ \ \text{where}\ 2 \leq k, p \leq m.  \label{conf_curv2}
    \end{align}
    By a change of variable from $s$ and $\tau$, it is direct to check that $R_{1\overline{1}1\overline{1}}$ also satisfies the Jacobi equation.
   \begin{equation}
    \rho''(s) +2R_{1\overline{1}1\overline{1}} \rho(s)=0, \ \ \rho(0)=0, \ \ \rho'(0)=1.  \label{Jaco_eqn2}
   \end{equation} Given $R_{1\overline{1}1\overline{1}}=-\frac{A}{(1+s^2)^{1+\frac{\varepsilon}{2}}}$ for $A>0$ small enough, we may solve $\rho(s)$ and obtain a complete Hermitian metric on $\mathbb{C}^m$. By \cite[Lemma 4.5, p.~62]{GW79} again, we have $s \leq \rho(s) \leq \eta s$ with 
   \(
   \eta =\exp (-\int_0^{+\infty} \frac{2A\,s}{(1+s^2)^{1+\frac{\varepsilon}{2}}}\,ds).
   \)
   By the same argument as in Lemma \ref{plane}, we have $s \leq C_1 (1+r)^{1+\varepsilon}$ for some constant $C_1>0$. It remains to show that there exists some constant $C_2>0$ so that
   \[
    BI(h) \leq -\frac{C_2}{(1+s)^{2+\varepsilon}}.
   \]
   By \eqref{conf_curv2}, $R_{2\overline{2}1\overline{1}}$ equals $R_{1\overline{1}1\overline{1}}$ at the origin, hence already negative near the origin. Moreover, it is negative for any $\tau>0$ as $\rho'(s)=1+2\tau \varphi'(\tau)>1$. As
   \[
   R_{1\overline{1}1\overline{1}}=-\frac{2(\tau \varphi^{\prime\prime}(\tau)+\varphi^{\prime}(\tau))}{e^{2\varphi}}=-\frac{A}{(1+s^2)^{1+\frac{\varepsilon}{2}}},
   \]
    we have
   \[
   2\tau \varphi^{\prime}(\tau)=\int_0^{\tau} \frac{A\,e^{2\varphi}}{(1+s^2)^{1+\frac{\varepsilon}{2}}}d\widetilde{\tau}.
   \]
   As $s \leq \rho(s) \leq \eta s$, we may find some constant $C_2>0$ so that
   \[
   -R_{2\overline{2}1\overline{1}}=\frac{1}{\tau e^{2\varphi}}  \int_0^{\tau}  \frac{A\,e^{2\varphi}}{(1+s^2)^{1+\frac{\varepsilon}{2}}}d\widetilde{\tau}=\frac{1}{\rho^2}  \int_0^{s}  \frac{2A\,e^{\varphi}\sqrt{\tau}}{(1+s^2)^{1+\frac{\varepsilon}{2}}}ds \geq \frac{1}{\eta^2 s^2} \int_0^s  \frac{2A\,s}{(1+s^2)^{1+\frac{\varepsilon}{2}}}ds. 
   \] We discuss the cases $0<s<1$ and $s \geq 1$ separately, and then extend the above estimate to the origin. Then there exists some constant $C_2$ so that 
    \[
    R_{2\overline{2}1\overline{1}} \leq -\frac{C_2}{(1+s)^2}.
    \]
    To sum up, the bisectional curvature formed by any two unitary vectors $U, V \in T^{1, 0} \mathbb{C}^m$ can be estimated by
    \[
    BI(U, V)=R_{11 k\overline{k}} U^1\overline{U^1} V^k\overline{V^k}+\sum_{p=2}^m R_{p \overline{p} k\overline{k}} U^p\overline{U^p} V^k\overline{V^k} \leq -\frac{C_3}{(1+s)^{2+\varepsilon}}.
    \]
\end{proof}

\subsection{Liouville theorems on holomorphic mappings}

Let $\mathcal{O}(M, N)$ denote the space of holomorphic mappings between two Hermitian manifolds $(M, g)$ and $(N, h)$. For any real number $\beta \geq 0$, let $r=d_g(p, \cdot)$ denote the geodesic distance from a fixed point $p \in M$ and $d_h$ is the geodesic distance from a fixed point on $N$. Recall that we have denoted by $r$ the Euclidean distance from the origin of $\mathbb{C}^n$ in Subsection \ref{sec_rota}. In the rest of Section \ref{Sec_3}, $r$ always denotes the geodesic distance from a fixed point on $(M, g)$. We introduce
\begin{align}
\mathcal{O}_{\beta}(M, N)=\bigr\{ &f \in \mathcal{O}(M, N)\ \ |\ \ \text{there exists}\ C>0\ \label{O_d_mapping}\\
      &\text{such that}\ d_{h} (f(q)) \leq C(1+r(q))^{\beta}\ \text{for all}\ q \in M  \bigr\}.  \nonumber
\end{align}

We restate Theorem \ref{grad_intro} as follows.

\begin{theorem}\label{grad}
Let $(M^n,g)$ be a complete Hermitian manifold with $Ric^{(2)}(g) \geq 0$. Assume that there exist constants $\delta \geq 0$ and $C_1>0$ such that the torsion of $(M, g)$ satisfies $|T| \le C_1(1+r)^{\delta}$ on $M$. Let $h_0$ denote the standard Euclidean metric on $\mathbb{C}$. For any given holomorphic function $f$, we define $u_0 \coloneqq \operatorname{tr}_g f^*h_0$. Given any $f \in \mathcal{O}_{\widetilde{\delta}}(M, g)$ with $\widetilde{\delta} \geq 0$ and any $\lambda>2\widetilde{\delta}-1+\delta$, there exists a constant $C_2>0$, depending on $\lambda$, $f$ and the geometry of $M$, such that
\begin{equation}
    u_0(q) \le C_2(1+r(q))^{\lambda},
    \qquad q\in M.
    \label{grade}
\end{equation} The same conclusion holds for holomorphic mappings in $\mathcal{O}_{\widetilde{\delta}}(M, \mathbb{C}^m)$ where $\mathbb{C}^m$ is endowed with the standard Euclidean metric.
\end{theorem}

The main ingredient of the proof of Theorem \ref{grad} is a Schwarz-type lemma for holomorphic mappings into a Hermitian manifold with a suitable negative curvature condition. For convenience, the constant $C$ may change from line to line in what follows.

\begin{lemma}\label{gene}
   Let $(M,g)$ be a complete Hermitian manifold with $Ric^{(2)}(g) \geq 0$ and its torsion satisfies $|T| \le C_1(1+r)^{\delta}$ for some $\delta \geq 0$. Assume that either of the following holds
   \begin{enumerate}
       \item   $(N, h)$ is a Hermitian manifold with $BI(h) \le -\frac{C_2}{(1+d_h)^{\alpha}}$ for some $\alpha \geq 0$ and $C_2>0$;
       \item   $(N, h)$ is a K\"ahler manifold with $H(h) \le -\frac{C_2}{(1+d_h)^{\alpha}}$ for some $\alpha \geq 0$ and $C_2>0$.
   \end{enumerate}
   Given any $f \in \mathcal{O}_{\beta}(M, N)$, there exists a constant $C_3>0$ such that
   \begin{equation}
        u(q) \coloneqq \operatorname{tr}_g f^*h \le C_3\,(1+r(q))^{\alpha\beta +\delta-1},\ \ \forall\,q \in M.
    \end{equation} 
In particular, $u \equiv 0$ when $\alpha\beta +\delta-1<0$.
\end{lemma}

\begin{proof}[Proof of Lemma \ref{gene}]
    First we consider the case $\alpha\beta +\delta-1\ge 0$. It suffices to show that if there exists some constant $\lambda$ such that  
    \begin{equation*}
        \limsup_{x\in M} \frac{u(x)}{(1+r)^{\lambda}} =\infty,
    \end{equation*}
    then $\lambda < \alpha\beta +\delta-1$. Assume that
    \begin{equation}
        \lambda \ge \alpha\beta +\delta-1\ge 0.\label{lambd}
    \end{equation}
    For some constant $a>0$ which is to be determined, we define 
    \begin{equation}
        v=1-(1+\frac{u}{(1+r)^{\lambda}})^{-a}.   \label{def_v}
    \end{equation}
    Then $0\le v <1$ and $\sup v=1$. Since $u=(1+r)^{\lambda} [(1-v)^{-1/a}-1]$, by a direct computation:
    \begin{align}
        \Box u=& 
        (1+r)^{\lambda-2}(\frac{1}{2}\lambda(\lambda-1) +\lambda (1+r) \Box r) [(1-v)^{-1/a}-1] \label{box1}   \\
        &+\frac{\lambda}{a} (1+r)^{\lambda-1}(1-v)^{-1/a -1} \langle \nabla v, \nabla r \rangle \nonumber \\
        &+ \frac{1+a}{a^2} (1+r)^{\lambda} (1-v)^{-1/a -2} \frac{1}{2}|\nabla v|^2  \nonumber \\
        &+\frac{1}{a} (1+r)^{\lambda} (1-v)^{-1/a -1} \Box v. \nonumber 
    \end{align}
    Thanks to Royden's lemma \cite{Royden1980}, in either case of $(N, h)$ we may apply the Chern-Lu inequality \eqref{chernlu2} and obtain
    \begin{equation}
        \Box u \ge \frac{C}{(1+r)^{\alpha \beta}}u^2. \label{laplace}
    \end{equation} 
    Therefore, for any given point on $M$, at least one of the four terms on the right hand side of \eqref{box1} must be bounded below by the right-hand side of \eqref{laplace}. In other words, given any sequence of points tending to infinity, one of the following inequalities necessarily holds along an infinite sequence of points on $M$. 
    \begin{align}
     & (1+r)^{\lambda-2}(\lambda(\lambda-1)\frac{1}{2} +\lambda (1+r) \Box r) ((1-v)^{-1/a}-1) \geq  \frac{C}{(1+r)^{\alpha \beta}}u^2.\label{box1a} \\
    &\frac{\lambda}{a} (1+r)^{\lambda-1}(1-v)^{-1/a -1} \langle \nabla v, \nabla r \rangle \geq  \frac{C}{(1+r)^{\alpha \beta}}u^2. \label{box2a} \\
    &\frac{1+a}{2a^2} (1+r)^{\lambda} (1-v)^{-1/a -2} |\nabla v|^2 \geq  \frac{C}{(1+r)^{\alpha \beta}}u^2.
    \label{box3a} \\
    &\frac{1}{a} (1+r)^{\lambda} (1-v)^{-1/a -1} \Box v \geq  \frac{C}{(1+r)^{\alpha \beta}}u^2. \label{box4a}
    \end{align}
    In order to obtain the desired upper bound estimate of $\lambda$, we analyze (\ref{box1a}), (\ref{box2a}), (\ref{box3a}), and (\ref{box4a}) along a suitably chosen sequence $\{x_k\} \subset M$ which tends to infinity. The idea of point-picking originated in the classical works of Omori \cite{Omori} and Cheng-Yau \cite[Theorem 3] {ChengYau1975}. We consider
\begin{equation}
    F_{k}(x)=\frac{v(x)}{[\log (2+r^2(x))]^{1/k}}, \ \text{where}\ k \in \mathbb{N}. \label{def_F_k}
\end{equation}
As $v$ is bounded, $F_k$ attains its maximum at some point $x_k \in M$. Hence
\begin{equation}
    \nabla F_k (x_k) =0,\ \ \ \  \Box F_k(x_k) \le 0. \label{max}
\end{equation}
We observe that
\begin{equation}
    \lim_{k \to +\infty} v(x_k)=\sup v.\label{mmax}
\end{equation}
Otherwise there exists a constant $\epsilon>0$ so that $v(x_k) +\epsilon \le \sup v -\epsilon$ holds for all but finitely many terms in the sequence $\{x_k\}$. Take $y \in M$ with $v(y) >\sup v -\epsilon$. Then $v(x_l)+\epsilon < v(y)$ for $l$ sufficiently large. Note that
\begin{equation*}
    \lim_{l \to +\infty} \Big( \frac{\log (2+r^2(y))}{\log 2}\Big)^{\frac{1}{l}}=1.
\end{equation*}
It follows that for $l$ sufficiently large,
\begin{align*}
    F_l(y)&=\frac{v(y)}{(\log (2+r^2(y)))^{1/l}}
    >\frac{v(x_l)+\epsilon}{(\log (2+r^2(y)))^{1/l}}
    \\
    &>\frac{v(x_l)}{(\log 2)^{1/l}} \ge \frac{v(x_l)}{[\log (2+r^2(x_l))]^{1/l}}=F_l(x_l).
\end{align*}
That is a contradiction. Hence \eqref{mmax} follows and the sequence $\{x_k\}$ tends to infinity. It also follows from \eqref{max} that 
\begin{align}
    &\nabla v(x_k)=\frac{2v(x_k)r\nabla r}{k(2+r^2)\log(2+r^2)}; \label{max1} \\
    &\Box v(x_k) \le \frac{(C+2r\Box r)v}{k(2+r^2)\log(2+r^2)}.  \label{max2}
\end{align}
If $x_k$ is not a cut point of $p\in M$, then we may apply Chen--Yang's complex Laplacian comparison theorem; see Lemma \ref{cLap_comp}. Thus we obtain
\begin{equation}
    \Box r \le \frac{n}{r}+C (1+r)^{\delta}.  \label{chenyanglap}
\end{equation}
Then together with $v\le 1$ we get from (\ref{max2}) that
\begin{equation}
    \Box v(x_k) \le \frac{C}{k (1+r)^{1-\delta} \log(2+r^2)}.  \label{chenyanglap2}
\end{equation}

Now we study the inequalities (\ref{box1a}),(\ref{box2a}),(\ref{box3a}), and (\ref{box4a}) along the (sub)sequence $\{x_k\}$ case by case. Recall that we have $(1-v)^{-1/a} \sim \frac{u}{(1+r)^{\lambda}} \rightarrow +\infty$ along $\{x_k\}$. 

Case 1: Assume (\ref{box1a}) holds along $\{x_k\}$. We use (\ref{chenyanglap}) and obtain
\begin{align*}
   &(1+r)^{\lambda-2} (1+r)^{1+\delta} \frac{u}{(1+r)^{\lambda}}  \geq  \frac{C}{(1+r)^{\alpha \beta}}u^2; \\
   &(1+r)^{\alpha \beta +\delta -1 -\lambda} \geq C \frac{u}{(1+r)^{\lambda}}.
\end{align*}
We get $\lambda< \alpha \beta+\delta-1$ after comparing the order of growth.

Case 2: Assume (\ref{box2a}) holds along $\{x_k\}$ and $0<a<1$. We apply (\ref{max1}) and argue similarly as in Case 1. It follows that $\lambda< \alpha \beta-2$.

Case 3: Assume (\ref{box3a}) holds along $\{x_k\}$. We apply (\ref{max1}) in the following estimates. 
\begin{align*}
   &(1+r)^{\lambda} [\frac{u}{(1+r)^{\lambda}}]^{1+2a} \frac{r^2}{(2+r^2)^2[\log(2+r^2)]^2}  \geq  \frac{C}{(1+r)^{\alpha \beta}}u^2; \\
   &\frac{(1+r)^{\alpha \beta -2 -\lambda}}{[\log (2+r^2)]^2} \geq C [\frac{u}{(1+r)^{\lambda}}]^{1-2a}.
\end{align*}
If we assume $0<a<\frac{1}{2}$, we obtain $\lambda< \alpha \beta-2$.

Case 4: Assume (\ref{box4a}) holds along $\{x_k\}$. We apply (\ref{chenyanglap2}) and estimate
\begin{align*}
   &\frac{1}{a}  (1+r)^{\lambda} [\frac{u}{(1+r)^{\lambda}}]^{1+a} \Box v  \geq  \frac{C}{(1+r)^{\alpha \beta}}u^2; \\
   &\frac{(1+r)^{\alpha \beta +\delta -1 -\lambda}}{\log (2+r^2)} \geq C [\frac{u}{(1+r)^{\lambda}}]^{1-a}.
\end{align*} If we assume $0<a<1$, then $\lambda< \alpha \beta+\delta-1$ holds.

Finally, we remark that (\ref{max1}), (\ref{max2}), (\ref{chenyanglap}), and (\ref{chenyanglap2}) still hold if $x_k$ lies on the cut locus of $p \in M$, thanks to Calabi’s trick in \cite{Calabi1958} and \cite[p.~342]{ChengYau1975}. 
To sum up, we may pick any $0<a<\frac{1}{2}$ in (\ref{def_v}) and obtain $\lambda< \alpha \beta+\delta-1$ since all cases contradict \eqref{lambd}. 

Next we consider the case $\alpha \beta+\delta-1<0$, assume that $u$ is unbounded and we pick $\lambda=0$ in \eqref{def_v}. Choosing $0 < a<\frac{1}{2}$, we get a contradiction again by the same argument. Hence $u$ is bounded. Then we consider 
\begin{equation}
    G_{k}(x)=\frac{u(x)}{[\log (2+r^2(x))]^{1/k}}.\label{def_G_k}
\end{equation}
As in \eqref{def_F_k}, we can pick $\{y_k\}$ such that $G_k(y_k)=\sup_M G$, and obtain 
\[
\Box u(y_k)
\le
\frac{Cu(y_k) (1+r(y_k))^{\delta-1}}{k\log(2+r^2(y_k))}.
\]
Together with \eqref{laplace}, we obtain 
\(
u(y_k)\le \frac{C}{k\log 2}.
\) If $\{y_k\}$ is bounded, then $u$ gets its maximum at some $y_0 \in M$. Hence $u=0$ by \eqref{box1} and the strong maximum principle. Otherwise 
it follows that $\sup u=\lim_{k \to +\infty} u(y_k) \le 0$. We conclude that $u\equiv 0$ in either case.

\end{proof}

\begin{proof}[Proof of Theorem \ref{grad}]
To prove Theorem \ref{grad}, for any $\lambda>2\widetilde{\delta}-1+\delta$, there exists $\varepsilon>0$, such that
\begin{equation*}
     (2+\varepsilon)(1+\varepsilon) \widetilde {\delta} +\delta-1 \le \lambda.
\end{equation*} 
Let $h$ be the Hermitian metric on $\mathbb{C}$ in Lemma \ref{plane}. We write 
\(
u= \operatorname{tr}_g f^*h=u_0 \xi,
\)
where $\xi$ is defined in (\ref{rs}). It follows from \eqref{xi(0)_formula} and \eqref{xi(0)_bound} that $u_0 \leq C u$ on $M$ for some $C$ which depends on $\epsilon$. Hence $(\ref{grade})$ follows from Lemma \ref{plane} and Lemma \ref{gene}. Similarly, the corresponding conclusion for holomorphic mappings from $M$ to $\mathbb{C}^m$ follows from Lemmas \ref{gene} and \ref{planetocm}. Alternatively, we may also apply \eqref{grade} on each component of the mapping.
\end{proof}

\begin{comment}
    As an application of Theorem \ref{grad}, we obtain the following Liouville-type results for holomorphic functions and holomorphic mappings. These results are motivated by Cheng's Liouville theorem for harmonic maps \cite{Cheng1980} and Chen--Yang's Schwarz lemma \cite{CY1984_2}.
    \begin{corollary}\label{Liouville}
Let $(M,g)$ be a complete Hermitian manifold with $Ric^{(2)}(g) \geq 0$. Assume that its torsion satisfies $|T| \le C(1+r)^{\delta}$ for constants $\delta \geq 0$ and $C>0$.   
\begin{enumerate}[label=(\arabic*)]
    \item  If $0 \leq \delta<1$ and $\widetilde{\delta} < \frac{1-\delta}{2}$, then $\mathcal{O}_{\widetilde{\delta}}(M, g)=\{\text{constant}\}$. 
    \item  If $\delta=1$, any bounded holomorphic function on $M$ is constant.
\end{enumerate}
\end{corollary}
\end{comment}

\begin{proof}[Proof of Corollary \ref{Liouville_intro}]
    In Case (1), by Theorem \ref{grad} we may choose some $\epsilon>0$ sufficiently small such that
    \(
    u \leq \frac{C}{(1+r)^{1-\delta-2\widetilde{\delta}-\epsilon}} 
    \) holds on $M$. Since $\Box u \geq 0$ by \eqref{chernlu2}, we may apply the strong maximum principle in view of Lemma \ref{Lapcompre}. It follows that $u$ is identically zero.

    In Case (2), (\ref{laplace}) is reduced to $\Box u \geq C u^2$. By Lemma \ref{gene}, $u$ is bounded on $M$. Motivated by \eqref{def_F_k} and \eqref{def_G_k}, we consider
    \[
    \widetilde{F}_k(x)=\frac{u(x)}{[\log (2+r^2(x))]^{1/k}}, \ \text{where}\ k \in \mathbb{N}.
    \]
    By the same argument in the proof of Lemma \ref{gene}, we get a sequence $\{\widetilde{x}_k\}$ so that a similar estimate as in (\ref{chenyanglap2}) holds.
    \[
    \Box u(\widetilde{x}_k) \leq \frac{C}{\log(2+(r(\widetilde{x}_k))^2)}.
    \]
    If a subsequence of $\{\widetilde{x}_k\}$ tending to infinity, then $\sup_M u=\lim_{k \rightarrow \infty} u(\widetilde{x}_k)=0$. If $\{\widetilde{x}_k\}$ is bounded, $u\equiv 0$ by the strong maximum principle. 
    
\end{proof}

We also record the following corollary, whose proof follows exactly the same idea as that of Lemma \ref{gene}.

\begin{corollary}\label{gene_coro}
Assume that the Hermitian manifold $(M, g)$ and the K\"ahler manifold $(N, h)$ satisfy the same assumptions as in Lemma \ref{gene}. Then the following conclusions hold.
\begin{enumerate}[label=(\arabic*)]
    \item  If $0\le \delta <1$ and $\alpha>0$, then $\mathcal{O}_{\beta}(M, N)=\{constant\}$ for every $\beta \leq \frac{1-\delta}{\alpha}$.
    \item   If $\delta=1$, then any bounded holomorphic map $f: M \rightarrow N$ is constant.
    \item   If $0 \le \delta \le 1$ and $\alpha=0$, then any holomorphic map $f: M \rightarrow N$ is constant. In particular, $M$ can not admit a K\"ahler metric with its holomorphic sectional curvature bounded above by a negative constant.
\end{enumerate}

\end{corollary}

\section{Complete Chern-flat manifolds with torsion of sublinear growth}\label{Sec_4}

In this section, we prove Corollary \ref{subline_intro} as an application of the gradient estimate in Theorem \ref{grad_intro}. Finally, Theorem \ref{func_Lie_summary} provides a complete characterization of such functions on complete simply connected Chern-flat manifolds with torsion of sublinear growth.

\subsection{Proof of Corollary \ref{subline_intro}}

\begin{proof}[Proof of Corollary \ref{subline_intro}]
Note that the universal cover $(\widetilde{M}, \widetilde{g})$ also satisfies the torsion growth condition. For simplicity, we assume that $M$ is simply-connected. By Lemma \ref{flat_para}, there exists a global unitary holomorphic frame $\{e_i\}_{i=1}^n$ on $M$. As in the proof of Theorem \ref{boothby_main}, it suffices to show that $T_{i j}^k$, the torsion component under $\{e_i\}$, is constant for any $1\le i,j,k \le n$. Note that $T_{i j}^k \in \mathcal{O}_{\delta}(M, g)$. It follows from Theorem \ref{grad} that 
\begin{equation*}
    |e_l(T_{ij}^k)|(x) \le C(1+d(x,p))^{\delta-\frac{1-\delta}{3}}.
\end{equation*}
That is, each time we take a derivative with respect to the frame $\{e_i\}$, we obtain a holomorphic function whose growth order is $\frac{1-\delta}{3}$-order smaller than that of the original function. For any integer $p$ with $(p+1)\frac{(1-\delta)}{3}>\delta$, we consider the holomorphic function $F=e_{l_p}\cdots e_{l_1}(T_{ij}^k)$ for any given $1 \leq l_1, \cdots, l_p \leq n$. 
Note that $e_{l_{p+1}}(F)$ must be constant for any $l_{p+1}=1,2,\cdots,n$. Hence $|\nabla F|^2=\sum_{j=1}^n |e_{j}(F)|^2$ is constant. Recall for any holomorphic function $\widehat{f}$, we have
\[
|\nabla \widehat{f}|^2 \coloneqq \sum_{k=1}^n |e_k(\widehat{f})|^2 =2|\nabla \bigl(\operatorname{Re}(\widehat{f})\bigr)|^2=2|\nabla \bigl(\operatorname{Im}(\widehat{f})\bigr)|^2.
\]
If $|\nabla F|^2$ is nonzero, we show that $F$ has at least linear growth by the following claim. This contradicts the fact that the growth order of $F$ does not exceed $\delta<1$. Hence $\nabla F=0$ and $F$ must be constant. It follows that the norm of gradient of $e_{l_{p-1}}\cdots e_{l_1}(T_{ij}^k)$ is constant, which implies that $F=0$ by the following claim again. By a repeated argument, we get that $e_{l_1}(T_{ij}^k)=0$ for any $1 \leq l_1\leq n$. Therefore, each $T_{ij}^k$ is constant.

To sum up, it remains to show the following result.

\vskip 0.2cm

\noindent \textbf{Claim}. Let $v$ be a real-valued smooth function on a complete noncompact manifold $(M,g)$. If  $|\nabla v|=c$ holds for a nonzero constant $c$, then $v$ has exactly linear growth in the sense that
\begin{equation}
        \limsup_{r(x)\to \infty} \frac{v(x)}{d_g(x,p)} =c. \label{exact_lin}
\end{equation}

\vskip 0.2cm

First of all, we show that the $``\geq"$ inequality in (\ref{exact_lin}) holds. Since $\nabla v$ is a nowhere-vanishing vector field, we can solve the integral curve $\gamma(t)$ such that $\gamma(0)=p$ and $\gamma'(t)=\frac{\nabla v (\gamma(t))}{|\nabla v (\gamma(t))|}$. We observe that $\gamma(t)$ is well-defined for any $t \in [0,+\infty)$ and 
    \begin{equation*}
       v(\gamma(t)) =v(p)+ct \to +\infty. 
    \end{equation*}
On the other hand, $d_g(\gamma(t),p)\le L(\gamma|_{[0, t]})=t$. Then $``\geq"$ inequality in (\ref{exact_lin}) follows if we evaluate the $\limsup$ expression in (\ref{exact_lin}) along $\gamma(t)$.

Next, we show $v(x) \leq v(p)+c d_g(x,p)$ for any $x \in M$. To see it, we pick a unit-speed minimizing geodesic $\delta(t)$ from $p$ to $x$. Then $\frac{d v(\delta(t))}{dt} \le |\nabla v (\delta(t))| \leq c$. The desired inequality follows after an integration.

\end{proof}

\begin{remark}\label{rem_sublinear}
There is a simpler proof of Corollary \ref{subline_intro} that bypasses Theorem \ref{grad}. It uses the holomorphic flows generated by the global holomorphic parallel frame provided by Lemma \ref{flat_para}. We refer to \cite{LWY2026} for further generalizations of Corollary \ref{subline_intro} along these lines. We nevertheless include the above proof in the hope that similar methods may apply beyond the Chern-flat setting.
\end{remark}

\subsection{\texorpdfstring{$\mathcal{O}_{d}(G)$}{O_d(G)} on complex nilpotent groups: existence}\label{reduce_nilpotent_sec}

\begin{comment}
Motivated by earlier works on harmonic functions of linear growth on complete manifolds with nonnegative Ricci curvature \cite{LT1989,Li1995,CCM1995}, we prove the following characterization of $\mathcal{O}_1(G)$.

\begin{corollary}\label{lin_space_char}
    Let $(G,g)$ be a simply connected complex Lie group with a left-invariant Hermitian metric. Assume $\operatorname{dim} G=n$. Let $H=[G,G]$ be the derived subgroup of $G$ and $k$ be the rank of the abelian group $G/H$. Then $G/H$, endowed with the induced metric from $g$, is holomorphically isometric to the complex Euclidean space $\mathbb{C}^k$ for some integer $k \geq 0$. Moreover, we have
    \begin{equation}
        \operatorname{dim} \mathcal{O}_1 (G)=k+1.  \label{dim_count_intro}
    \end{equation} 
    In particular, $(G, g)$ is holomorphically isometric to the complex vector group $\mathbb{C}^n$ with the flat metric if and only if $\operatorname{dim} \mathcal{O}_1 (G)=n+1$.
\end{corollary}    
\end{comment}

\begin{lemma}\label{liegrouplem}
    Let $G$ be a simply connected complex Lie group and $H=[G,G]$ be the derived group which is generated by all commutators of $G$. Then $H$ is a connected closed normal subgroup of $G$. Moreover, $G/H$, if not trivial, is isomorphic to the vector group $\mathbb{C}^k$ for some $k \in \mathbb{N}$.
\end{lemma}

\begin{proof}[Proof of Lemma \ref{liegrouplem}]
    First of all, we observe that $H$ is connected once $G$ is connected. Indeed, the set of all commutators of $G$ is connected in $G$. So is $H$ as it is generated by a connected subset containing the identity $\mathbf{1}_G \in G$.
    Assume that $G$ is simply connected with its Lie algebra $\mathfrak{g}$. Then $H$ is a connected Lie subgroup of $G$ with Lie algebra $[\mathfrak{g}, \mathfrak{g}]$. Moreover, it is a closed normal subgroup of $G$. See \cite[p.~138]{Hoch_B} or \cite[Theorem W.2.2]{Conrad2018} for a complete proof. Then $\pi:G \to G/H$ has a $H$-principal bundle structure. Consider the long exact sequence of homotopy groups
    \begin{equation*}
        \cdots \to \pi_1(H) \to \pi_1(G) \to \pi_1(G/H) \to \pi_0(H) \to \cdots.
    \end{equation*}
    Since $\pi_1(G)=0$ and $\pi_0(H)=0$, $G/H$ is simply connected. The only simply connected abelian complex Lie group is the vector group $\mathbb{C}^k$.
\end{proof}

\begin{proof}[Proof of Theorem \ref{poly_space_char_intro}]
    By the same argument in the proof of Lemma \ref{liegrouplem}, $\widehat{H}$ is a closed connected normal subgroup of $G$, and $G/\widehat{H}$ is simply connected. Consider the projection $\pi: G \rightarrow G/\widehat{H}$. We need to show that $\pi^{\ast}: \mathcal{O}_d(G/\widehat H) \to \mathcal{O}_d(G)$ is an isomorphism between vector spaces. Note that \begin{equation}
     d_{\widehat{g}}(\pi(x),\pi(y))=\inf_{\widetilde{y} \in \pi^{-1}(\pi(y))}d_g(x,\widetilde{y}) \le d_g(x,y). \label{dcontrol}
\end{equation} In view of \eqref{dcontrol}, $\pi^{\ast}$ is well-defined and injective. It remains to show $\pi^{\ast}$ is surjective. Let $\mathfrak{g}$ and $\widehat{\mathfrak{h}}$ denote the Lie algebras of $G$ and $\widehat{H}$, respectively. We choose $\{e_1,\cdots,e_n\}$ as a global unitary left-invariant holomorphic frame on $(G, g)$ such that $\{e_{p+1}, \cdots,e_n\}$ is the corresponding frame on $\widehat{\mathfrak{h}}$. For any given $f \in \mathcal{O}_d(G)$, we show that it is constant along each fiber of $\pi$. Indeed, by the definition of $\widehat{H}$, for any $p+1 \leq j \leq n$, we may write $e_j$ as a linear combination of $[\cdot, [\cdot, [\cdots]]]$ of a sufficiently large length. Applying Theorem \ref{grad_intro} repeatedly, we get that $e_j(f) \equiv 0$. Thus we may assume $f=\pi^{\ast}\widehat{f}$ for some $\widehat{f} \in \mathcal{O}(G/\widehat{H})$. According to \eqref{dcontrol} that the induced metric $\widehat{g}$ on $G/\widehat{H}$ is a quotient metric, we estimate
    \[
    |\widehat{f}([x])|=|f(x)|=|f(\widetilde{x})| \leq C(1+d_g(\widetilde{x}, \mathbf{1}_G ))^d=C(1+d_{\widehat{g}}([x], [\mathbf{1}_G]))^d.
    \] Here we choose $\widetilde{x} \in [x]$ such that $d_{\widehat{g}}([x], [\mathbf{1}_G])=d_g(\widetilde{x}, \mathbf{1}_G)$. To sum up, we have proved $\pi^{\ast}$ is surjective. Finally, if there exist no non-constant holomorphic functions with polynomial growth on $(G, g)$, then $\mathfrak{g}$ is perfect. Otherwise we can find holomorphic functions of linear growth on $G$ by considering the vector group $G/H$ where $H=[G, G]$. Indeed, let $\mathfrak{g}$ and $\mathfrak{h}=[\mathfrak{g},\mathfrak{g}]$ denote the Lie algebras of $G$ and $H$ respectively. Recall that $\{e_1,\cdots,e_n\}$ is a globally unitary left-invariant holomorphic frame on $(G, g)$. Moreover, we may assume that $\{e_{k+1}, \cdots,e_n\}$ is the corresponding frame on $\mathfrak{h}$. Then the left-invariant metric $g$ descends to a left-invariant Hermitian metric $\widehat{g}$ on $G/H$. Let $\widehat{\pi}: G \to G/H$ be the projection map. Let $\widehat{e}_i=\widehat{\pi}_{\ast} e_i$ for $i=1,2,\cdots,k$ denote the induced frame on $G/H$. Since $[e_i, e_j] \in \mathfrak{h}$, we get that $[\widehat{e}_i,\widehat{e}_j]=\widehat{\pi}_{\ast}[e_i, e_j]=0$. It is direct to check that $d\widehat{\omega}=0$. In other words, $\{\widehat{e}_1, \cdots,\widehat{e}_k\}$ is a unitary left-invariant holomorphic frame with respect to $\widehat{g}$ and $(G/H, \widehat{g}$) is holomorphically isometric to the Euclidean space $\mathbb{C}^k$. In view of this identification, we may choose the global coordinate functions $z_1,\ldots,z_k$ of linear growth on $G/H$. The corresponding $\widehat{\pi}^{\ast}(z_i)$ for $1 \leq i \leq k$ are nonconstant holomorphic functions of linear growth on $G$.
\end{proof}

Let $G$ be a simply connected, nilpotent, complex Lie group and $\mathfrak{g}$ its Lie algebra. Assume that $\mathfrak{g}$ is $s$-step nilpotent for some integer $s\geq 1$. According to the lower central series in \eqref{l_central_def}, we have
\begin{equation}\label{l_central_pf}
    \mathfrak{g}=\mathfrak{g}_1 \supset \mathfrak{g}_2 \supset \cdots \supset\mathfrak{g}_s \supset \mathfrak{g}_{s+1}=\{0\}.
\end{equation}
We may check $[\mathfrak{g}_{i}, \mathfrak{g}_{j}] \subset \mathfrak{g}_{i+j}$. Now we construct a basis of $\mathfrak{g}$ inductively. First choose a basis of $\mathfrak g_s$ and introduce
\begin{equation}
V_s=\operatorname{span}\{e_{(s,1)},\ldots,e_{(s,n_s)}\}.       \label{adapted_index}
\end{equation}
Next, for each $k=s-1,\ldots,1$, we introduce
\[
V_k=\operatorname{span}\{e_{(k,1)},\ldots,e_{(k,n_k)}\} \subset \mathfrak g_k,\ \ \ n_k=\dim_{\mathbb C}\mathfrak g_k-\dim_{\mathbb C}\mathfrak g_{k+1}, 
\]
such that the images of $\{e_{(k,1)},\ldots,e_{(k,n_k)}\}$ form a basis of $\mathfrak g_k/\mathfrak g_{k+1}$. Note that $\mathfrak{g}_k=V_k \oplus \mathfrak{g}_{k+1}$ for any $1 \leq k \leq s-1$. In particular, we obtain a decomposition
\begin{equation} \label{adapt_decomp}
\mathfrak{g}=V_1 \oplus \cdots \oplus V_s.  
\end{equation}
Such a decomposition of $\mathfrak g$ is called an \emph{adapted decomposition} (with respect to the lower central series \eqref{l_central_pf}). 

\begin{lemma}[See, for example, {\cite[Exercise 10.5.6]{LeDonne_B}}]\label{adapted_lemma}
Let $\mathfrak{g}$ be an $s$-step nilpotent Lie algebra with the lower central series \eqref{l_central_pf}. For $r\geq 1$, define
\[
L_r(V_1)
\coloneqq\operatorname{span}\left\{
[x_1,[x_2,\ldots,[x_{r-1},x_r]]]
\ |\ \ x_1,\ldots,x_r\in V_1
\right\}.
\]
Then there exists an adapted decomposition such that
\[
V_r \subseteq L_r(V_1), \ 1 \leq r < s,\ \ \ \  V_s=L_s(V_1), 
\]
and
\[
[V_i, V_j] \subset V_{i+j}\oplus V_{i+j+1}\oplus\cdots\oplus V_s,\ \ \ 1 \leq i+j \leq s.
\]
\end{lemma}

For a simply connected nilpotent complex Lie group $G$, the exponential map $\exp:\mathfrak{g} \to G$ is a biholomorphism, see \cite[Theorem 1.2.1]{CG1990} or \cite[Theorem 1.104]{Knapp2002} for a detailed proof. The above choice of the basis of $\mathfrak{g}$ produces global coordinate functions $(z^1,\cdots,z^n)$ on $G$ such that
\begin{equation}\label{exp_holo}
\exp (z^1(q) e_1+ \cdots+ z^n(q) e_n)=q \in G.    
\end{equation}
Let $X=z^{i}e_i \in \mathfrak{g}$ and $dX=dz^i \otimes e_i$. Then the Maurer--Cartan form in terms of exponential coordinates can be expressed as a finite series.
\begin{equation}\label{series}
\sum_{k\geq 0}
\frac{(-1)^k}{(k+1)!}\operatorname{ad}_X^k(dX)=\sum_{k=0}^{s-1} \frac{(-1)^k}{(k+1)!} ad_X^k dX.    
\end{equation}
Here we use the fact that $\mathfrak{g}$ is $s$-step nilpotent. For a proof of \eqref{series}, see \cite[Theorem 2.14.3]{Vara_B_1984} for general Lie groups, and \cite[Theorem 5.4]{Hall_B} for matrix groups. Write
\(
[e_i,e_j]=c_{ij}^k e_k,
\)
where $c_{ij}^k$ denotes the structure constants of $\mathfrak g$ with respect to the basis $\{e_1,\ldots,e_n\}$. We may write down the left-invariant coframe with respect to $\{e_i\}$ based on \eqref{series}.
\begin{equation}\label{nil_coframe}
    \varphi^i=dz^i -\frac{1}{2} c_{jk}^i z^j dz^k +\frac{1}{6} c_{pq}^i c_{jk}^q z^p z^j dz^k+ \ldots,\ \ \ 1\leq i, j, k, p, q \leq n.
\end{equation}
Now we consider the left-invariant (Chern-flat) Hermitian metric on $G$ in the form of
\begin{equation}\label{nil_metric}
    \omega=\sqrt{-1}(\varphi^1 \wedge \overline{\varphi^1}+\cdots+\varphi^n \wedge \overline{\varphi^n}).
\end{equation}

The following result estimates the growth order of coordinate functions of the exponential map with respect to \eqref{nil_metric}.

\begin{proposition}\label{coordi_grow}
    Let $G$ be a simply connected complex Lie group such that its Lie algebra $\mathfrak{g}$ is $s$-step nilpotent. Choose any $p=(k, \alpha)$ with $1 \leq k \leq s$ and $1 \leq \alpha \leq n_k$. Here we follow the notation in \eqref{adapted_index}. Then we have $z^p \in \mathcal{O}_k(G)$. 
\end{proposition}

\begin{comment}
    Moreover, for any $\epsilon>0$, we have $z^p \notin \mathcal{O}_{k-\epsilon}(G)$ for $p=(k, \alpha)$ where $2 \leq k \leq \min(3, s)$.
\end{comment}

\begin{proof}[Proof of Proposition \ref{coordi_grow}]

%We divide the proof into two main steps.
%\textbf{Step 1:}  We prove that $z^p \in \mathcal{O}_k(G)$ for any $1\le k \le s$. 

We proceed by an induction on $k$. Here $p$ denotes the pair $(k,\alpha)$ with $1 \leq \alpha \leq n_k$. 

We first show that $z^p$ has linear growth for every $p=(1,\alpha)$ with $1\leq \alpha\leq n_1$. For any $C^1$ curve $\gamma(t)=(z^1(t),\ldots,z^n(t))$ with $t\in [0,1]$ from the origin to $(z^1,\ldots,z^n)$, we have
\begin{equation}
    L_g(\gamma(t)) \geq \int_\gamma |\varphi^p| =\int_0^1 |(z^p(t))'|dt \geq \sup_{\gamma(t)} |z^p(t)|
\end{equation}
 for $(1, 1) \leq p \leq (1, n_1)$. Taking the infimum over the lengths of such curves $\gamma$, we conclude that $z^p$ has at most linear growth. Moreover, by restricting to the straight line from the origin to $(0,\ldots,z^p,\ldots,0)$, one sees that $z^p$ has exactly linear growth. 

Assume that 
\begin{equation}\label{l_control}
    \int_0^1 |(z^p(t))'|dt \le C\bigl(1+L_g(\gamma(t)\bigr)^l,\ \ \ \text{for}\ p=(l, \alpha)\ \text{with}\ 1 \leq l \leq k-1\ \text{and}\ 1\leq \alpha \leq n_{l}.
\end{equation}
Now we consider $p=(k,\cdot)$, we rewrite $(\ref{nil_coframe})$ as
\begin{equation*}
    \varphi^p= dz^p + \sum_{A=1}^{k-1}  \Xi_A.
\end{equation*}
Let $k_1,k_2,\cdots,k_A, l_1,l_2,\cdots,l_A$ be integers taking values in $\{1,2,\cdots,k-1\}$. Assume that $k_1+l_1 \le k$, $k_2+l_2 \le l_1$, $\cdots$, and $k_A+l_A \le l_{A-1}$. Then $\Xi_A$ can be written as a linear combination of terms in the following form
\begin{equation}\label{term_gen}
    c^{(k,\alpha)}_{(k_1,\alpha_1)(l_1,\beta_1)} c^{(l_1,\beta_1)}_{(k_2,\alpha_2)(l_2,\beta_2)} \cdots 
    c^{(l_{A-1},\beta_{A-1})}_{(k_A,\alpha_A)(l_A,\beta_A)}
    z^{(k_1,\alpha_1)} z^{(k_2,\alpha_2)} \cdots z^{(k_A,\alpha_A)} dz^{(l_A,\beta_A)}.
\end{equation}
Note that we use Lemma \ref{adapted_lemma} repeatedly.

Fix any $p=(d_1, \cdot)$ and $q=(d_2, \cdot)$ such that $d_1+d_2 \leq k$. According to the induction hypothesis, we have $z^p \in \mathcal{O}_{d_1}(G)$ and $z^q \in \mathcal{O}_{d_2}(G)$. Indeed, by the triangle inequality, we obtain the following estimate. Let $\gamma(t)=(z^1(t),\ldots,z^n(t))$ with $t \in [0, 1]$ be any $C^1$ curve joining the identity $e$ and $(z^1,\ldots,z^n) \in G$. By the induction hypothesis \eqref{l_control}, we have
\[
\max_{t \in [0, 1]}|z^p|(\gamma(t)) \leq \int_0^1 |(z^p(t))'|dt \leq C\bigl(1+L(\gamma(t))\bigr)^{d_1}, \ \ \int_0^1 |(z^q(t))'|dt \leq C\bigl(1+L(\gamma(t))\bigr)^{d_2}.
\]
There exists some constant $C>0$ such that
\begin{equation}
    \int_0^1 |z^p(t) (z^q(t))'| dt \le C\bigl(1+L_g(\gamma(t))\bigr)^{d_1+d_2}.
\end{equation}
Then we obtain 
\begin{equation}
    \int |\Xi_A(\gamma(t))| \le C(1+L_g(\gamma(t))^{k_1+k_2+\cdots+k_A+l_A}) \le C\bigl(1+L_g(\gamma(t))\bigr)^k.
\end{equation}
For any $p=(k, \cdot)$, we obtain
\begin{equation}
    |z^p|\le \int_0^1 |(z^p(t))'|dt  \le \int |dz^p + \sum_{A=1}^{k-1}  \Xi_A| + \sum_{A=1}^{k-1} \int |\Xi_A| \le L_g(\gamma(t)) + C\bigl(1+L_g(\gamma(t))\bigr)^k.
\end{equation}
Therefore $z^p \in \mathcal{O}_k(G)$.

\end{proof}

\begin{proof}[Proof of Corollary \ref{nilpotent_alg_intro}]
    It suffices to prove the ``if'' direction as the ``only if'' part is direct from Proposition \ref{coordi_grow}. Assume that $\widehat{H}$ in Theorem \ref{poly_space_char_intro} is nontrivial. Then we construct a left-invariant holomorphic frame $\{e\}$ along $\widehat{H} \subset G$. As in the proof of Theorem \ref{poly_space_char_intro}, for any $d\geq 0$ and $f\in\mathcal{O}_d(G)$, we have $e(f)\equiv 0$. In other words, the Jacobian of any $n$ holomorphic functions $f_1,\ldots,f_n\in\mathcal{O}_d(G)$ is everywhere degenerate on $G$. Therefore, $\widehat{H}$ is trivial and $G$ is nilpotent.
\end{proof}

\subsection{\texorpdfstring{$\mathcal{O}_{d}(G)$}{O_d(G)} on complex nilpotent groups: a complete characterization}\label{subsec_nilpot}

\begin{definition}[Carnot algebra {\cite[p.~254]{LeDonne_B}}]\label{Carnot_alg}
An $s$-step nilpotent Lie algebra $\mathfrak g$ is called a Carnot algebra, or a nilpotent Lie algebra with a stratification, if it admits an adapted decomposition \eqref{adapt_decomp}
such that
\begin{equation} \label{adapt_Carnot}
[V_1,V_r]=V_{r+1}, \ \  1 \leq r<s, \qquad [V_1,V_s]=\{0\}.
\end{equation} In particular, $V_r = L_r(V_1)$ for any $1 \leq r \leq s$ and $[V_i, V_j] \subseteq V_{i+j}$ for $1 \leq i+j \leq s$. Such an adapted decomposition is called a stratification.
\end{definition}

\begin{definition}[Carnot group {\cite[Definition 11.1.1 on p.~326]{LeDonne_B}}]\label{Carnot_group}
A simply connected Lie group $G$ whose Lie algebra is a Carnot algebra is called a Carnot group.
\end{definition}

Let $G$ be a simply connected complex Carnot group whose Lie algebra $\mathfrak{g}$ is an $s$-step Carnot algebra. Fix a left-invariant Hermitian metric $g$ on $G$. Recall $L_x$ is a left translation by $x \in G$. Given $V_1 \subset \mathfrak{g}$, we obtain a left-invariant subbundle $\mathcal{V} \subset TG$ such that $\mathcal{V}_x=(L_{x})_{\ast}(V_1)$. We call $\mathcal{V}$ a \emph{horizontal subbundle} and a $C^1$ curve $\gamma(t)$ with $t \in [0, 1]$ on $G$ a \emph{horizontal curve} if $\gamma'(t) \in \mathcal{V}$ for any $t \in [0, 1]$, that is to say, $\gamma'(t) \in (L_{\gamma(t)})_{\ast}(V_1)$. Now that the restriction of $g$ on $V_1 \subset \mathfrak{g}$ provides a natural left-invariant inner product $g_1$ on $\mathcal{V}$. We can define the \emph{Carnot--Carath\'eodory distance} between two points $p, q \in G$.
\begin{equation}\label{CC_def}
    d_{CC}(p, q)=\inf \Bigl\{ \int |\gamma'(t)|_{g_1} \ | \ \gamma(t)\ \text{is a horizontal curve joining}\ p\ \text{and}\ q \Bigr\}.
\end{equation}
Obviously $d_g(p,q) \le d_{CC}(p,q)$ holds for any $p, q \in G$. As a special case of a fundamental result due to Chow \cite[Theorem 4.1.8]{LeDonne_B}, $d_{CC}$ is a geodesic distance and induces the manifold topology on $G$. Now we recall the ball-box theorem on complex Carnot groups.

\begin{lemma}[{\cite[Theorem 11.2.3 on p.~337]{LeDonne_B}}]\label{ballbox}
Let $G$ be a simply connected complex Carnot group equipped with a left-invariant Hermitian metric $g$. Let $d_{CC}$ denote the Carnot--Carath\'eodory distance induced by $g$ in the sense of \eqref{CC_def}. Fix a basis $\{e_1, \dots, e_n\}$ of $\mathfrak{g}$ adapted to the stratification $V_1 \oplus \cdots \oplus V_s$. Let $\exp: \mathfrak{g} \to G$ be the exponential map and $(z^1, \ldots, z^n)$ be the corresponding global coordinates with respect to $\{e_1, \cdots, e_n\}$, 
\begin{equation*}
\exp\left(\sum_{j=1}^{n} z^j(q) e_j\right)=q \in G.
\end{equation*}
According to \eqref{adapted_index}, we assume that each index $1 \leq i \leq n$ can be written as $(d_i, \alpha_{i})$ where $1 \leq \alpha_i \leq n_{d_i}$. Here integers $d_i$ with $1\leq i \leq n$ exhaust all integers from $1$ to $s$. Let $\mathtt{Box}(r)$ denote the following anisotropic polydisk with respect to $\{e_1, \ldots, e_n\}$.
\begin{equation*}
\mathtt{Box}(r) \coloneqq \left\{ (z_1, \dots, z_n) \in \mathbb{C}^n\ \ |\ \ |z_j| < r^{d_j} \right\}, \quad r > 0.
\end{equation*}
Then, there exists a constant $C > 1$ such that for all $r > 0$,
\begin{equation}\label{geo_control}
B_{CC}(\mathbf{1}_G, \frac{r}{C}) \subset \exp(\mathtt{Box}(r)) \subset B_{CC}(\mathbf{1}_G, C r),
\end{equation}
where $B_{CC}(\mathbf{1}_G, R)$ denotes the $d_{CC}$ metric ball of radius $R$ centered at the identity element $\mathbf{1}_G \in G$. 
\end{lemma}

In our application, it is crucial that the inequality \eqref{geo_control} holds on a Carnot group at every scale $r>0$. Combining the Cauchy estimate, we get a dimension estimate on holomorphic functions of polynomial growth with respect to $d_{CC}$. 
\begin{proposition}\label{Odcc_char}
Let $G$ be a simply connected complex Carnot group and $\mathcal{O}_d^{CC} (G)$ denote the space of all $f \in \mathcal{O}(G)$ which are of polynomial growth order at most $d$ with respect to $d_{CC}$. Then
\begin{equation}\label{O_dcc_basis}
\mathcal{O}_d^{CC} (G) \subset \operatorname{span}\Bigl\{ (z^1)^{\alpha_1}\cdots (z^n)^{\alpha_n} \ \ \ |\ \ \ \alpha_1 d_1+\alpha_2 d_2 +\cdots+\alpha_n d_n \le d  \Bigl\}.
\end{equation}
\end{proposition}

\begin{proof}[Proof of Proposition \ref{Odcc_char}]
Given any $f \in \mathcal{O}_d^{CC} (G)$, we apply the Cauchy estimate on the polydisk $\prod_{j=1}^n \{|z^j|<r^{d_j}\}$. It follows from $(\ref{geo_control})$ that
\begin{equation}
    |D^\alpha f(\mathbf{1}_G)| \le \frac{\alpha!}{r^{\alpha_1 d_1+\alpha_2 d_2 +\cdots+\alpha_n d_n}}\, \sup_{B_{CC}(\mathbf{1}_G, C r)} |f|  \le C\frac{(1+r)^d}{r^{\alpha_1 d_1+\alpha_2 d_2 +\cdots+\alpha_n d_n}}.
\end{equation}
Then $D^\alpha f(\mathbf{1}_G)=0$ for any multi-index $\alpha$ satisfying $\alpha_1 d_1+\alpha_2 d_2 +\cdots+\alpha_n d_n>d$.
\end{proof}

Note that $\mathcal{O}_d(G) \subset \mathcal{O}_d^{CC} (G)$. On the other hand, it follows from Proposition \ref{coordi_grow} that 
\[
(z^1)^{\alpha_1}(z^2)^{\alpha_2}\cdots(z^n)^{\alpha_n} \in \mathcal{O}_d(G)\ \ \text{if}\ \alpha_1 d_1+\alpha_2 d_2 +\cdots+\alpha_n d_n \leq d. 
\]
Therefore, we have determined $\mathcal{O}_d(G)$ for any simply connected complex Carnot group.

\begin{corollary}\label{O_d_Carnot}
Let $G$ be a simply connected complex Carnot group with a left-invariant Hermitian metric $g$. Let $z^1, \ldots, z^n$ be the exponential coordinates and $d_1, \ldots, d_n$ as in  Lemma \ref{ballbox}. Then for any $d \ge 0$, 
    \begin{equation}
        \mathcal{O}_d(G)=
        \operatorname{span} \Bigl\{(z^1)^{\alpha_1}(z^2)^{\alpha_2}\cdots(z^n)^{\alpha_n}\ \ | \ \ \ \alpha_1 d_1+\alpha_2 d_2 +\cdots+\alpha_n d_n \le d \Bigr\}.
    \end{equation}    
\end{corollary}

\begin{definition}[Associated Carnot algebra {\cite[Definition 9.2.16]{LeDonne_B}}]\label{ass_carnot}
Let $G$ be a simply connected Lie group with its Lie algebra $\mathfrak{g}$ being $s$-step nilpotent. Consider the lower central series of $\mathfrak{g}$
\begin{equation}
    \mathfrak{g}=\mathfrak{g}_1 \supset \mathfrak{g}_2 \supset \cdots \supset\mathfrak{g}_s \supset \mathfrak{g}_{s+1}=\{0\},\ \ \  \text{with}\ \mathfrak{g}_{k+1}=[\mathfrak{g}, \mathfrak{g}_k],\ k \geq 1.
\end{equation}
The Carnot algebra $\mathfrak{g}_{\infty}$ is defined by
\(
\mathfrak{g}_{\infty} =\bigoplus_{j=1}^s \mathfrak{g}_j /\mathfrak{g}_{j+1},
\)
equipped with the Lie bracket $[[ \cdot, \cdot]]_{\infty}$ such that
\begin{equation}\label{bracket_infinity}
    [[X+\mathfrak{g}_{i+1}, Y+\mathfrak{g}_{j+1}]]_{\infty}=[X, Y]+\mathfrak{g}_{i+j+1}.
\end{equation}
Note that $\mathfrak{g}_{\infty}$ has a natural stratification (compare \eqref{adapt_decomp} and Definition \ref{Carnot_alg}) such that its lower central series $\mathfrak{g}_{\infty}=(\mathfrak{g}_{\infty})_1 \supset (\mathfrak{g}_{\infty})_2 \supset \cdots (\mathfrak{g}_{\infty})_s \supset (\mathfrak{g}_{\infty})_{s+1}=\{0\}$ satisfies
\begin{align*}
(\mathfrak{g}_{\infty})_i=(V_{\infty})_i \oplus \ldots  \oplus (V_{\infty})_s,\ \ \text{where}\
(V_{\infty})_i=\mathfrak{g}_i /\mathfrak{g}_{i+1}, \ \ 1 \leq i \leq s.
\end{align*}
In particular, we have $[(V_{\infty})_i, (V_{\infty})_j] \subset (V_{\infty})_{i+j}$ for any $1 \leq i+j \leq s$.
\end{definition}

Let $G$ be a simply connected Lie group with an $s$-step nilpotent Lie algebra with a left-invariant Riemannian metric $g$.
The associated Carnot algebra ${\mathfrak{g}}_{\infty}$ satisfies $(V_{\infty})_1=\mathfrak{g}_1 /\mathfrak{g}_2=\mathfrak{g} /[\mathfrak{g},\mathfrak{g}]$, which is the abelianization of $\mathfrak{g}$. Let $\pi_{ab}$ denote the corresponding quotient map $\mathfrak{g} \rightarrow \mathfrak{g} /[\mathfrak{g},\mathfrak{g}]$. We introduce a quotient norm (called Pansu limit norm in {\cite[Section 12.4]{LeDonne_B}}) on $(V_{\infty})_1$ as follows.
\begin{equation}\label{ab_norm}
    \|v\|_{ab}=\min \bigl\{  \|w\|_g\ \ |\ \ w \in \mathfrak{g},\ \pi_{ab}(w)=v  \bigr\},\ \ v \in (V_{\infty})_1.
\end{equation}
Note that $(V_{\infty})_1$ generates $\mathfrak{g}_{\infty}$ after repeated applications of Lie bracket. Let $G_{\infty}$ be the simply connected Carnot group with its Lie algebra $\mathfrak{g}_{\infty}$. Note that $(V_{\infty})_1$ generates a natural horizontal subbundle on $G_{\infty}$. From \eqref{ab_norm}, we may introduce a Carnot-Carath\'eodory distance function $d_{\infty}$ (called Pansu limit metric in {\cite[Section 12.4]{LeDonne_B}}) on $G_{\infty}$. 

Given a simply connected nilpotent Lie group $G$ with a left-invariant Riemannian metric $g$, then $g$ induces a Riemannian distance $d_g$ on G. There is a remarkable connection between $d_{CC}$ on $G$ and $d_{\infty}$ on $G_{\infty}$. Namely, Pansu \cite{Pansu1983} (see also \cite{BL2013}) proved that $(G_{\infty}, d_{\infty})$ is the asymptotic cone of $(G, d_{g})$. We refer to \cite[Section 12]{LeDonne_B} for more information. Now we assume that $\mathfrak{g}$, the Lie algebra of $G$, is $s$-step nilpotent. Consider its associated Carnot algebra ${\mathfrak{g}}_{\infty}$ and the corresponding Carnot group $G_{\infty}$. Assume that both $\mathfrak{g}$ and ${\mathfrak{g}}_{\infty}$ admit adapted decompositions as follows.
\begin{align*}
&\mathfrak{g}=V_1 \oplus \cdots \oplus V_s,\ \ \  [V_i, V_j] \subset V_{i+j}\oplus V_{i+j+1}\oplus\cdots\oplus V_s,\ \ \ 1 \leq i+j \leq s .\\
&\mathfrak{g}_{\infty}=(V_{\infty})_1 \oplus \cdots  \oplus (V_{\infty})_s,\ \ \  [[(V_{\infty})_i, (V_{\infty})_j]]_{\infty} \subset (V_{\infty})_{i+j},\ \ \ \ 1 \leq i+j \leq s.
\end{align*}
Moreover, for each $1\leq i\leq s$, the vector spaces $V_i$ and $(V_{\infty})_i$ are isomorphic. Let $\rho: \mathfrak{g} \rightarrow \mathfrak{g}_{\infty}$ denote a complex linear isomorphism which preserves the above decomposition. Recall the exponential maps $\exp_{G}: \mathfrak{g} \rightarrow G$ and $\exp_{G_{\infty}}: \mathfrak{g}_{\infty} \rightarrow G_{\infty}$ are global diffeomorphisms. We have a global diffeomorphism $\mathcal{F}: G \rightarrow G_{\infty}$ such that
\begin{equation}\label{def_F_cal}
\mathcal{F}(\exp_{G}(X))=\exp_{G_{\infty}}(\rho(X)),\ \ \ X \in \mathfrak{g}.    
\end{equation}

The following result shows under the above identification $\mathcal{F}$, the distance from identity on $G$ and $G_{\infty}$ are equivalent.

\begin{lemma}[Guivarc'h {\cite{Guivarch1973}, see also \cite[Corollary 12.4.4]{LeDonne_B}}]
     Let $G$ be a simply connected nilpotent complex Lie group with a left-invariant Hermitian metric $g$. Let $\mathcal{F}$ be the diffeomorphism defined in \eqref{def_F_cal}, and $\mathbf{1}_G$ and $\mathbf{1}_{G_{\infty}}$ be the identity elements of $G$ and $G_{\infty}$ respectively. There exists $C>1$ such that for any $x\in G$,
    \begin{equation}\label{equiva_identity}
        \frac{1}{C} d_{\infty}(\mathbf{1}_{G_{\infty}}, \mathcal{F}(x))-C \le d_{g}(\mathbf{1}_G,x) \le C d_{\infty}(\mathbf{1}_{G_{\infty}}, \mathcal{F}(x))+C.
    \end{equation}
\end{lemma}

In our case, $G$ and $G_{\infty}$ are complex Lie groups. The above $\mathcal{F}$ is a biholomorphism. We conclude that $\mathcal{F}^{\ast}: \mathcal{O}_{d}(G_{\infty}, d_{\infty}) \rightarrow \mathcal{O}_{d}(G)$ is a linear isomorphism. On the other hand, Proposition \ref{Odcc_char} characterizes $\mathcal{O}_{d}(G_{\infty}, d_{\infty})$. Together with Theorem \ref{poly_space_char_intro} and Proposition \ref{coordi_grow}, this yields the following result for $\mathcal{O}_d(G)$. Note that parts (1) and (2) are restatements of Theorem \ref{poly_space_char_intro} and Theorem \ref{func_summary_intro}. Part (3) follows from the proof of Theorem \ref{poly_space_char_intro}.

\begin{theorem}\label{func_Lie_summary}
    Let $(G, g)$ be a simply connected complex Lie group equipped with a left-invariant Hermitian metric. Suppose that the lower central series of its Lie algebra $\mathfrak{g}$ defined in \eqref{l_central_def} stabilizes at the $s$-th term $\mathfrak{g}_s$. Let $\widehat{H}$ be the connected Lie subgroup of $G$ with Lie algebra $\mathfrak{g}_s$. Similarly, let $\widehat{H}_k$ denote the connected Lie subgroup generated by $\mathfrak{g}_k$. Then we have the following conclusions.
    \begin{enumerate}[label=(\arabic*)]
        \item $\mathcal{O}_d (G)$ is isomorphic to $\mathcal{O}_d (G/\widehat{H})$.
        \item $\widehat{G}=G/\widehat{H}$ is a simply connected complex Lie group and its Lie algebra $\widehat{\mathfrak{g}}=\mathfrak{g}/\mathfrak{g}_{s}$ is $(s-1)$-step nilpotent. When $s=1$, $G=\widehat{H}$ and $\mathcal{O}_d (G)$ contains only constants. Assume $\operatorname{dim}\widehat{G}=n$ and $s>1$. We choose an adapted basis $\{e_1, \ldots, e_n\}$ of its Lie algebra in the sense of \eqref{adapt_decomp}. According to \eqref{adapted_index}, we assume that each index $1 \leq i \leq n$ can be written as $(d_i, \alpha_{i})$ where $1 \leq \alpha_i \leq n_{d_i}$. Here these integers $d_i$ with $1\leq i \leq n$ exhaust all integers from $1$ to $s-1$. Let $\exp: \widehat{\mathfrak{g}} \to \widehat{G}$ be the exponential map and $(z^1, \ldots, z^n)$ be the corresponding global coordinates with respect to $\{e_1, \cdots, e_n\}$. Then
        \begin{equation}
             \mathcal{O}_d(\widehat{G})=
        \operatorname{span} \Bigl\{(z^1)^{\alpha_1}(z^2)^{\alpha_2}\cdots(z^n)^{\alpha_n}\ \ |\ \ \alpha_i \in \mathbb{N} \cup \{0\}, \ \ \alpha_1 d_1+\alpha_2 d_2 +\cdots+\alpha_n d_n \le d \Bigr\}.
        \end{equation}
       \item $\mathcal{O}_d (G)$ is isomorphic to $\mathcal{O}_d (G/\widehat{H}_k)$ where $k=\lfloor d \rfloor +1$.
    \end{enumerate}
\end{theorem}

\section{Examples}\label{Sec_5}

This section has two purposes. In the first half, we construct examples of complete Hermitian manifolds satisfying $Ric^{(2)} \geq 0$ mentioned in Example \ref{sharp_L_intro} and Example \ref{sharp_G_intro}. In the second half, we present examples of complete Chern-flat Hermitian manifolds in low dimensions. For example, the examples constructed in Proposition \ref{2dim_more} show that Corollary \ref{subline_intro} is sharp. We also discuss examples related to Questions \ref{ques3_intro} and \ref{ques4_intro}.

\subsection{Construction of Example \ref{sharp_L_intro}}\label{sharp_L_ex}

Let $M$ denote $\mathbb{C}^2 \times \mathbb{D}$. We consider the following Hermitian metric on $M$.
\begin{equation}
  \omega=\sqrt{-1} H(t) (dz_1 \wedge d\overline{z_1} +  dz_2 \wedge d\overline{z_2})+\sqrt{-1} \frac{1}{(1-|w|^2)^2} dw \wedge d\overline{w},  \label{product_1} 
\end{equation}
where 
\begin{equation}
 H(t)=\frac{1}{(1+t)(1+\log(1+t))^{2\alpha}},\ \ \ \alpha \in (\frac{1}{2}, 1)\,\text{ and }\,t=|z|^2=|z_1|^2+|z_2|^2.  \label{def_H(t)}   
\end{equation}
First of all, we observe that $\omega$ defined in \eqref{product_1} is complete. Since $\omega$ is the product metric of a $U(2)$-invariant metric on $\mathbb{C}^2$ and the Poincar\'e metric on $\mathbb{D}$, it suffices to check that the radial distance in the $\mathbb{C}^2$ part has an infinite length, which amounts to showing $\int_0^{\infty} \frac{\sqrt{H(t)}}{2\sqrt{t}}dt=\infty$. Indeed, for the product metric $\omega$, the distance function $d_{g}$ can be estimated by
\begin{equation}\label{dist_product}
d_g^2 \leq C_1 \frac{[\log(1+|z|^2)]^{2-2\alpha}}{(1-\alpha)^2} +C_2 \Bigl[\log \frac{1+|w|}{1-|w|}\Bigr]^2.    
\end{equation}
Next we solve the Chern curvature tensor of $\omega$. Consider the following unitary frame $\{e_i\}$ and its corresponding coframe
\[
\varphi^1=\sqrt{H} dz_1,\ \ \  \varphi^2=\sqrt{H} dz_2,\ \ \ \ \varphi^3=\frac{1}{1-|w|^2}dw.
\] By the structural equations \eqref{struct1}, we solve the connection form and torsion forms as follows.
\begin{align}
&\theta=\begin{pmatrix}
    \widetilde{\alpha}  &  0  &  0\\
     0  &  \widetilde{\alpha}   &  0\\
    0  &  0  &  \widetilde{\beta}
\end{pmatrix},\ \ \  \widetilde{\alpha}=\frac{H'}{2(\sqrt{H})^3} \sum_{j=1}^2 (\overline{z_j} \varphi^j-z_j\overline{\varphi^j}).\ \ \ \widetilde{\beta}=\overline{w} \varphi^3-w \overline{\varphi^3};\\
&(\tau^1, \tau^2, \tau^3)=(-\frac{H'}{(\sqrt{H})^3}\overline{z_2} \varphi^1 \wedge \varphi^2, \frac{H'}{(\sqrt{H})^3}\overline{z_1} \varphi^1 \wedge \varphi^2, 0).   \label{torsion_product}
\end{align}
We proceed to solve the curvature form by \eqref{struct2}.
\begin{align*}
&\Theta=d\theta=\begin{pmatrix}
    d\widetilde{\alpha}  &  0  &  0\\
     0  &  d\widetilde{\alpha}   &  0\\
    0  &  0  &  d\widetilde{\beta}
\end{pmatrix},\ \ \ \ A=\frac{(\log H)'}{H},\ \ \ \ \ \ B=\frac{(\log H)^{\prime\prime}}{H}.\\
&  d\alpha=-A \sum_{j=1}^2 \varphi^j \wedge \overline{\varphi^j}-B(\sum_{j=1}^2 \overline{z_j}\varphi^j) \wedge (\sum_{j=1}^2 z_j \overline{\varphi^j}),\ \ \ \ d\beta=-2 \varphi^3 \wedge \overline{\varphi^3}.
\end{align*}
For any conformal Hermitian metric $\widehat{\omega}=e^{2u(t)}\omega$, the corresponding torsion and curvature forms can be solved by, for example, \cite[Formula (36), p.~1215]{YZ2018}. Under the corresponding unitary frame $\{\widehat{e}_i=e^{-u} e_i\}$ where $1 \leq i \leq 3$, the torsion and curvature components satisfy
\begin{equation}
\widehat{T}_{jk}^i=e^{-u} [T_{jk}^i+e_j(u)\delta_{ik}-e_k(u)\delta_{ij}], \ \ \ \widehat{\Theta}_{i}^j=\Theta_i^j-2\partial \overline{\partial} u\,E.      \label{torsion_curv_conf}
\end{equation}
Now we solve the second Ricci curvature of the conformal metric $\widehat{\omega}$. Note that the $\mathbb{C}^2$ part of the metric is $U(2)$-invariant, it suffices to solve for a point $(z_1, z_2, w)=(\sqrt{t}, 0, w)$. Now we have 
\[
\partial \overline{\partial} u=\frac{(tu'(t))'}{H} \varphi^1 \wedge \overline{\varphi^1}+\frac{u'(t)}{H}\varphi^2 \wedge \overline{\varphi^2}.
\]
The corresponding second Ricci curvature has a diagonal form with
\begin{align}
&\widehat{Ric}^{(2)}_{1\overline{1}}=\widehat{Ric}^{(2)}_{2\overline{2}}=
e^{-2u} [\frac{-2(u'+(tu'(t))')}{H}-(2A+tB)],   \label{2nd_ricci_12}\\
&\widehat{Ric}^{(2)}_{3\overline{3}}=e^{-2u} [\frac{-2(u'+(tu'(t))')}{H}-2].  \label{2nd_ricci_3}
\end{align}

Finally, we choose a suitable conformal factor $u(t)$ such that $\widehat{\omega}$ is the desired Hermitian metric. We write $2u=-3U$ and define 
\begin{equation}
U(t)=\int_0^t \frac{1}{\tau^2}  \int_0^{\tau} sH(s)ds d\tau.  \label{def_U}
\end{equation}
This particular form of $U(t)$ is motivated by the equation $(t^2 U'(t))^{\prime}=tH(t)$, which is related to the form $u'+(tu'(t))'$ in the expression of $\widehat{Ric}^{(2)}_{3\overline{3}}$ in \eqref{2nd_ricci_3}.
Note that $U(t)$ is a smooth non-decreasing nonnegative function on $[0, \infty)$ with only zero at $t=0$. Moreover, it is uniformly bounded as
\begin{equation}
    U(\infty)=\int_0^{\infty} sH(s) \int_s^{\infty} \frac{1}{\tau^2}  d\tau=\int_{0}^{\infty} H(s)ds=\frac{1}{2\alpha-1}.
\end{equation}
Therefore $\widehat{\omega}$ is quasi-isometric to $\omega$, hence complete. We observe that
\begin{align}
&\frac{-2(u'+(tu'(t))')}{H}-2=3\frac{(t^2 U^{\prime}(t))'}{tH}-2=1,\\
&\frac{-2(u'+(tu'(t))')}{H}-(2A+tB)=3+\frac{\Bigl[ \frac{t^2}{1+t}+\frac{2\alpha t^2}{(1+t)(1+\log (1+t))}\Bigr]^{\prime}}{tH} > 3.    \label{2_Ricci_eigen}
\end{align}
Therefore, we have proved that $\widehat{Ric}^{(2)} (\widehat{\omega}) \geq 1$. It remains to estimate the growth of torsion. By \eqref{dist_product}, \eqref{torsion_product}, and \eqref{torsion_curv_conf}, the nonzero torsion components of $\widehat{\omega}$ at the point $(\sqrt{t}, 0, |w|)$, with respect to the frame $\{\widehat{e}_i\}$, are the following.
\begin{align*}
    \widehat{T}_{12}^2=-\widehat{T}_{21}^2=e^{\frac{3}{2}U} \big[\frac{H'\sqrt{t}}{2H^{\frac{3}{2}}}-\frac{3U' \sqrt{t}}{2\sqrt{H}}\big],\ \ \ \widehat{T}_{13}^3=-\widehat{T}_{31}^3=-e^{\frac{3}{2}U} \frac{3U'(t) \sqrt{t}}{2\sqrt{H}}.
\end{align*}
Therefore 
\begin{equation*}
|\widehat{T}|^2 = 2e^{3U} \frac{t}{H} \Big[ 9(U')^2+\big[(\log H)'-3U'\big]^2\Big].
\end{equation*}
We can show by a direct calculation that $|\widehat{T}|(q) \leq C_3(\alpha) (1+d_{\widehat{g}}(q))^{\frac{\alpha}{1-\alpha}}$. Moreover, the exponent $\frac{\alpha}{1-\alpha}$ can not be lowered further. As $\alpha \rightarrow \frac{1}{2}+$, we get a family of Hermitian metrics required in Example \ref{sharp_L_intro}.

\subsection{Construction of Example \ref{sharp_G_intro}}\label{sharp_G_ex}

We consider the following Hermitian metric on $\mathbb{C}^3$.
\begin{equation}
  \omega_0=\sqrt{-1} H(t) (dz_1 \wedge d\overline{z_1} +  dz_2 \wedge d\overline{z_2})+\sqrt{-1} dw \wedge d\overline{w},  \label{product_2} 
\end{equation}
where $t=|z_1|^2+|z_2|^2$ and $H(t)$ is defined in \eqref{def_H(t)} except that we assume $\alpha \in (0, 1)$. From the calculation on Example \ref{sharp_L_intro}, we know 
\[
d_{{g}_0}(\cdot)  \leq C (1+\log(1+t))^{1-\alpha}+|w|,\ \ |T|(g_0) \leq C (1+d_{g_0}(\cdot))^{\frac{\alpha}{1-\alpha}},\ \ Ric^{(2)}(g_0) \geq 0.
\] However, $w \in \mathcal{O}_1(\mathbb{C}^3, g_0)$ and $|\nabla w| \equiv 1$. Thus, the exponent in the gradient estimate cannot be decreased by the desired amount. We set
\begin{equation}\label{def_s_dis}
\delta=\frac{\alpha}{1-\alpha},\ \ L=1+\log(1+t),\ \  \ s=5+L^{1-\alpha}.    
\end{equation}
Our goal is to introduce a new Hermitian metric by modifying the coefficient before the $dw \wedge d\overline{w}$.
\begin{equation}
  \omega=\sqrt{-1} H(t) (dz_1 \wedge d\overline{z_1} +  dz_2 \wedge d\overline{z_2})+\sqrt{-1} B(s, |w|)dw \wedge d\overline{w},  \label{modify_product} 
\end{equation}
Assume that the growth of the distance and torsion of $\omega$ were the same as that of $\omega_0$, we introduce
\begin{align}
    B(s, |w|)=s^{-(1+\delta)} e^{P(s, |w|)},\ \ \ P \geq 0,  \ \ \ r=|w|. \label{B_def_1}
\end{align} We hope that by a suitable choice of $P(s, |w|)$, $w$ has nearly linear growth with respect to $d_{g}$ while $|\nabla w|^2=\frac{1}{B}=s^{1+\delta} e^{-P} \leq C (1+d_g)^{1+\delta}$. Moreover, $Ric^{(2)}(\omega)$ is positive everywhere.

Consider the following unitary frame of $\omega$
\[
\varphi^1=\sqrt{H}dz_1, \ \ \varphi^2=\sqrt{H}dz_2,\ \ \varphi^3=\sqrt{B}dw.
\]
We solve the structure equations of \eqref{modify_product} using \eqref{B_def_1}. If we set 
\[
u=\log H,\ \ \ F=\log B=-(1+\delta)\log s+P.
\]
Then the connection and torsion forms are 
\begin{align}
&\theta=\begin{pmatrix}
    \frac{1}{2}(\partial u-\overline{\partial} u)  &  0  &  0\\
     0  &  \frac{1}{2}(\partial u-\overline{\partial} u)   &  0\\
    0  &  0  &  \frac{1}{2}(\partial F-\overline{\partial} F)
\end{pmatrix};\\
&(\tau^1, \tau^2, \tau^3)=\Bigl(-\frac{u'}{\sqrt{H}}\overline{z_2} \varphi^1 \wedge \varphi^2, \frac{u'}{\sqrt{H}}\overline{z_1} \varphi^1 \wedge \varphi^2, \frac{s'F_s}{\sqrt{H}}(\overline{z_1} \varphi^1 \wedge \varphi^3+\overline{z_2} \varphi^2 \wedge \varphi^3)\Bigr).   \label{connect_torsion_B}
\end{align}
We may solve the norm of the torsion as follows.
\begin{equation}\label{torsion_norm_B}
    |T|^2=\frac{t}{H}((u')^2+(s'F_s)^2)=\frac{t}{1+t} \Bigl[ L^{2\alpha} (1+\frac{2\alpha}{L})^2+(1-\alpha)^2 (P_s-\frac{1+\delta}{s})^2  \Bigr]
\end{equation}
Note that the curvature form is also diagonal.
\begin{align}
   \Theta_1^{1}=&\Theta_2^{2}=-\frac{1}{H} \sum_{1 \leq i, j \leq 2}(u'\,\delta_{ij}+u^{\prime\prime} \overline{z_i}z_j) \varphi^i \wedge \varphi^j,\\
   \Theta_3^{3}=&-\frac{1}{H} \sum_{1\le i.j \leq 2} \Bigl(s'F_s \delta_{ij}+\bigl[(s')^2F_{ss}+s^{\prime\prime}F_s\bigr]\overline{z_i}z_j\Bigr) \varphi^i \wedge \overline{\varphi^j}-\frac{1}{4B}(F_{rr}+\frac{F_r}{r})\varphi^3 \wedge \overline{\varphi^3}\\
   &-\frac{s'F_{sr}}{2|w|\sqrt{HB}}\sum_{1 \leq i \leq 2}  (\overline{z_i} w \varphi^i \wedge \overline{\varphi^3} + z_i \overline{w} \varphi^3 \wedge \overline{\varphi^i})  \nonumber
\end{align}
Therefore we solve $Ric^{(2)}$ from the connection forms. Since $H(t)$ and $B(s, |w|)$ only depend on $|z|$ and $|w|$, it suffices to solve at a point $(\sqrt{t}, 0, |w|)$.
\begin{align}
    &Ric^{(2)}_{1\overline{1}}=Ric^{(2)}_{2\overline{2}}=-\frac{2u'+tu^{\prime\prime}}{H}>0,\ \ \text{compared with } 
    \eqref{2_Ricci_eigen},\\
    &Ric^{(2)}_{3\overline{3}}=-\frac{2s'F_s+\bigl[(s')^2F_{ss}+s^{\prime\prime}F_s\bigr]t}{H}-\frac{F_{rr}+\frac{F_r}{r}}{4B}.
\end{align}
After we plug into $F=-(1+\delta)\log s +P$, the formula for $Ric^{(2)}_{3\overline{3}}$ can be further simplified. In order to show the positivity of $Ric^{(2)}$, it suffices to show
\begin{align}\label{ricci_posi_condi}
I(t, |w|)\coloneqq C(t)-\frac{(1-\alpha)^2 t}{1+t} P_{ss}-\frac{2s'+ts^{\prime\prime}}{H} P_s-\frac{s^{1+\delta}}{4e^P}(P_{rr}+\frac{P_r}{r})>0,   
\end{align} where 
\begin{equation}\label{c(t)_plus_term}
    \begin{split}
      &C(t)\coloneqq\frac{1}{1+t} \Bigl[\frac{L^{\alpha}}{s} \bigl (t+2-\frac{\alpha t}{L}\bigr)-\frac{(1-\alpha) t}{s^2}\Bigr]=\frac{L^{\alpha} s(t+2-\frac{\alpha t}{L})-(1-\alpha) t}{(1+t) s^2},\\
      &\frac{2s'+ts^{\prime\prime}}{H}=\frac{(1-\alpha)L^{\alpha}}{1+t}(t+2-\frac{\alpha t}{L}).
   \end{split}
\end{equation}

A natural thought is that we try $P \equiv 0$. However, the holomorphic coordinate $w$ has the growth order $\frac{3+\delta}{2}$ in this case. To that end, we consider 
\begin{equation}
P(s, |w|)=N\log (1+\frac{c|w|^2}{N s^2})    \label{P_def_final}
\end{equation}
for a sufficiently small constant $c>0$ and a positive integer $N$. Note that our choice of $P$ tends to $\frac{c\,r^2}{s^2}$ as $N \rightarrow \infty$.

\noindent \textbf{Claim}: Fix $\delta=\frac{\alpha}{1-\alpha} \in (0, 1)$, or equivalently, $\alpha \in (0, \frac{1}{2})$. Choose any $0<c<\frac{5(1-\alpha)}{6^{\frac{1}{1-\alpha}}}$ and any positive integer $N$, each Hermitian metric $\omega$ defined by \eqref{modify_product} and \eqref{P_def_final} is complete and satisfies 
\begin{equation}
    |T| \leq C_1 (1+d_g)^{\delta},\ \ \ Ric^{(2)}(g)>0,\ \ \ |w| \leq C_2 (1+d_g)^{\frac{2N+3+\delta}{2N+2}},\ \ \ \ |\nabla w|^2 \leq C_3(1+d_g)^{1+\delta}.  \label{sharp_property_est}
\end{equation} for constants $C_1, C_2$, and $C_3$ depending on $N, \delta$. Moreover, all the exponents in these estimates cannot be lowered further. Letting $N \rightarrow \infty$, we get the desired sequence of Hermitian metrics in Example \ref{sharp_G_intro}. 

Let $d_g(z, w)$ denote the distance from $(z, w) \in \mathbb{C}^3$ to the origin. We first show a key estimate needed in the proof of the above claim, Namely, there exists positive constants $C_1$ and $C_2$ such that for any $(\widetilde{z}, \widetilde{w}) \in \mathbb{C}^3$, 
\begin{align}
     C_1 \Bigl(\frac{(1+\log (1+|\widetilde{z}|))^{1-\alpha}-1}{1-\alpha}+\min\bigl(|\widetilde{w}|^{\frac{2N+2}{2N+3+\delta}}, |\widetilde{w}|\bigl) \Bigr)\leq  d_g(\widetilde{z}, \widetilde{w}) \leq  C_2 \Bigr(s(t(\widetilde{z}))+|\widetilde{w}|^{\frac{2N+2}{2N+3+\delta}}\Bigr). \label{P_dis_est}
\end{align}

\vskip 0.1cm

For simplicity, we assume that the constants $C$ might change from line to line in the following.

\vskip 0.1cm

\noindent \textbf{Step 1: Proof of \eqref{P_dis_est}.} Let $\gamma (\tau)$ denote any piecewise $C^1$ curve $(z(\tau), w(\tau))$ for $\tau \in (\tau_0, \tau_1)$ from the origin to $(\widetilde{z}, \widetilde{w})$. Then its length is
\begin{equation}\label{P_Length_est1}
L_g(\gamma(\tau)|_{[\tau_0, \tau_1]}) =\int_{\tau_0}^{\tau_1}  \sqrt{H(|z(\tau)|^2)|z'(\tau)|^2 +B(s(z(\tau)), |w(\tau)|) |w'(\tau)|^2} d\tau.
\end{equation}

The proof of the upper bound in \eqref{P_dis_est} is motivated by methods in sub-Riemannian geometry, where we use a suitable combination of horizontal or vertical paths to minimize the length of $\gamma (\tau)$.

We first assume $|\widetilde{w}|>1$. We consider a curve from $(0, 0)$ to $(z^{\ast}, 0)$ horizontally in $\mathbb{C}^2$, $(z^{\ast}, 0)$ to $(z^{\ast}, \widetilde{w})$ vertically, and $(z^{\ast}, \widetilde{w})$ to $(\widetilde{z},\widetilde{w})$ horizontally. Note that $\mathbb{C}^2$ part of the metric is $U(2)$-invariant. We may assume that two horizontal curves go along the radial direction in $\mathbb{C}^2$. Here we pick $z^{\ast}$ such that
\[
s(z^{\ast})=M \coloneqq \max (s(t(\widetilde{z})), |\widetilde{w}|^{\frac{2N+2}{2N+3+\delta}}).
\]
If $M=s(t(\widetilde{z}))$, then $z^{\ast}=\widetilde{z}$ and the second horizontal curve is empty. We estimate the length of the first horizontal curve by
\[
\int_0^{t(\widetilde{z})}  \frac{\sqrt{H(t)}}{2\sqrt{t}}dt \leq s(t(\widetilde{z})).
\]
If $M=|\widetilde{w}|^{\frac{2N+2}{2N+3+\delta}}$, then $|z^{\ast}| \geq |\widetilde{z}|$ and the second horizontal curve goes backwards. The sum of the lengths of two horizontal curves is
\[
2\int_0^{t(z^{\ast})} \frac{\sqrt{H(t)}}{2\sqrt{t}}dt -\int_0^{t(\widetilde{z})}  \frac{\sqrt{H(t)}}{2\sqrt{t}}dt \leq 2s(t(z^{\ast}))-s(t(\widetilde{z})) \leq 2 |\widetilde{w}|^{\frac{2N+2}{2N+3+\delta}}.
\]
The vertical curve from $(z^{\ast}, 0)$ to $(z^{\ast}, \widetilde{w})$ can be chosen such that it goes linearly along the radial direction in $\mathbb{C}$. The corresponding length can be estimated
\begin{align*}
   &s(|z^{\ast}|)^{-\frac{1+\delta}{2}} \int_0^{|\widetilde{w}|}   \Bigl(1+\frac{c \chi^2}{N s^2} \Bigr)^{\frac{N}{2}}d\chi\\
   \leq &  C(N, c)  [s(|z^{\ast}|)]^{-\frac{1+\delta}{2}}  \int_0^{|\widetilde{w}|} \Bigl[1+(\frac{c}{N})^{\frac{N}{2}} \frac{\chi^N}{[s(|z^{\ast}|)]^N} \Bigl]d\chi,\ \ \  \text{we use } (1+ \lambda x)^N \leq C(\lambda, N) (1+x^N)\\
   \leq &  C(N, c) \Bigl(\frac{|\widetilde{w}|}{M^{\frac{1+\delta}{2}}} +  \frac{|\widetilde{w}|^{N+1}}{M^{N+\frac{1+\delta}{2}}} \Bigr)\\
   \leq &  C |\widetilde{w}|^{\frac{2N+2}{2N+3+\delta}}.
\end{align*}
Note that $M \geq |\widetilde{w}|^{\frac{2N+2}{2N+3+\delta}}$, therefore the last step follows from
\begin{align*}
    & \frac{|\widetilde{w}|^{N+1}}{M^{N+\frac{1+\delta}{2}}} \leq {|\widetilde{w}|}^{N+1-\frac{2N+2}{2N+3+\delta}\frac{2N+1+\delta}{2}}= {|\widetilde{w}|}^{\frac{2N+2}{2N+3+\delta}},   \\   
     & \frac{|\widetilde{w}|}{M^{\frac{1+\delta}{2}}} 
    \leq |\widetilde{w}|^{1-\frac{2N+2}{2N+3+\delta} \frac{1+\delta}{2}} \leq |\widetilde{w}|^{\frac{2N+2}{2N+3+\delta}},\ \ \ \text{where we use }\ |\widetilde{w}|>1.
\end{align*}

So far, we have proved the upper bound in \eqref{P_dis_est} in the case $|\widetilde{w}|>1$. If $|\widetilde{w}| \leq 1$, we consider a curve from $(0, 0)$ to $(0, \widetilde{w})$ vertically, then $(0, \widetilde{w})$ to $(\widetilde{z}, \widetilde{w})$ horizontally. By a direct calculation, we can show the length of this curve is bounded by $s(t(\widetilde{z}))+C|\widetilde{w}| \leq C s(t(\widetilde{z}))$. To sum up, we have proved the upper bound in \eqref{P_dis_est} for any $(\widetilde{z}, \widetilde{w}) \in \mathbb{C}^3$.

We proceed to estimate the lower bound in \eqref{P_dis_est}. We consider any curve $\gamma (\tau)$ where $\tau \in (\tau_0, \tau_1)$ from the origin to $(\widetilde{z}, \widetilde{w})$. Let $L(\gamma(\tau))$ be the length on the interval $(\tau_0, \tau)$. As the $\mathbb{C}^{2}$ part of \eqref{modify_product} is $U(2)$-invariant, it follows from \eqref{P_Length_est1} that
\begin{equation}\label{dist_lower_1}
L(\gamma(\tau)) \geq \frac{1}{2}\Bigl(\int_0^{t(z(\tau))} \frac{\sqrt{H(t)}}{2\sqrt{t}}dt+
\int_{\tau_0}^{\tau_1} [s(z(\tau))]^{-\frac{1+\delta}{2}}  (1+\frac{c |w(\tau)|^2}{N [s(z(\tau))]^2})^{\frac{N}{2}} |w'(\tau)| d\tau  \Bigr)
\end{equation}
Note that $R(\tau) \coloneqq \int_0^{t(z(\tau)} \frac{\sqrt{H(t)}}{2\sqrt{t}}dt$ can be estimated explicitly. 
\[
R(\tau) \geq \int_0^{\tau} \frac{1}{2(1+\tau) (1+\log (1+\tau))^{\alpha}} d\tau =\frac{(1+\log (1+\tau))^{1-\alpha}-1}{1-\alpha}.
\]
We may also compare $R(\tau)$ with $s(\tau)$. As 
\[
\frac{ds}{dR}=\frac{(1-\alpha) L^{-\alpha} \frac{1}{1+t}} { \frac{1}{2\sqrt{t}} \frac{1}{\sqrt{t+1} L^{\alpha}} } \leq 2(1-\alpha). 
\] It follows that
\begin{equation}\label{s_R_relation}
s(\tau) \leq 6+2(1-\alpha) R(\tau) \leq 6+4(1-\alpha) L(\gamma(\tau)).    
\end{equation}
With the aid of this formula, we could estimate the second term in the right-hand side of \eqref{dist_lower_1}. 
\begin{equation}\label{latter_harder_est}
\begin{split}
    &\int_{\tau_0}^{\tau_1} [s(z(\tau))]^{-\frac{1+\delta}{2}}  (1+\frac{c |w(\tau)|^2}{N [s(z(\tau))]^2})^{\frac{N}{2}} |w'(\tau)| d\tau   \\
    \geq &  \frac{1}{C}  [6+4(1-\alpha) L(\gamma(\tau_1))]^{-N-\frac{1+\delta}{2}} \int_{\tau_0}^{\tau_1} |w(\tau)|^N |w'(\tau)| d\tau  \\
    \geq &  \frac{1}{C}  [6+4(1-\alpha) L(\gamma(\tau_1))]^{-N-\frac{1+\delta}{2}} \int_{0}^{|\widetilde{w}|} |w|^N d|w|,\ \ \text{where we use }  \frac{d|w(\tau)|}{d\tau} \leq |w'(\tau)|,  \\   
    \geq &  \frac{1}{C}  [6+4(1-\alpha) L(\gamma(\tau_1))]^{-N-\frac{1+\delta}{2}}  |w|^{N+1}. 
\end{split}
\end{equation}
To sum up, we show 
\begin{equation}\label{dist_lower_2}
|\widetilde{w}|^{N+1} \leq C L(\gamma(\tau_1)) [6+4(1-\alpha) L(\gamma(\tau_1))]^{N+\frac{1+\delta}{2}} \end{equation}
If $L(\gamma(\tau_1)) \geq 1$, by \eqref{dist_lower_1} and \eqref{dist_lower_2} we get that
$|\widetilde{w}|^{N+1} \leq C L(\gamma(\tau_1))^{N+1+\frac{1+\delta}{2}}$. Together with \eqref{dist_lower_1}, we conclude that
\begin{equation}\label{dist_lower_3}
L(\gamma(\tau_1)) \geq C \max \Bigl(|\widetilde{w}|^{\frac{2N+2}{2N+3+\delta}}, \frac{(1+\log (1+t(\widetilde{z})))^{1-\alpha}-1}{1-\alpha}\Bigr).    
\end{equation}

If $L(\gamma(\tau_1)) < 1$, we may carry out (\ref{latter_harder_est}) in a simpler way.
\begin{equation}\label{latter_harder_est_2}
\begin{split}
    &\int_{\tau_0}^{\tau_1} [s(z(\tau))]^{-\frac{1+\delta}{2}}  (1+\frac{c |w(\tau)|^2}{N [s(z(\tau))]^2})^{\frac{N}{2}} |w'(\tau)| d\tau   \\
    \geq &   [6+4(1-\alpha) L(\gamma(\tau_1))]^{-\frac{1+\delta}{2}} \int_{\tau_0}^{\tau_1}  |w'(\tau)| d\tau  \\
    \geq &    C |\widetilde{w}|,\ \ \text{where we use } \,L(\gamma(\tau_1)) < 1.
\end{split}
\end{equation} By \eqref{dist_lower_1} and \eqref{latter_harder_est_2}, we get
\begin{equation}\label{dist_lower_4}
L(\gamma(\tau_1)) \geq  \frac{(1+\log (1+t(\widetilde{z})))^{1-\alpha}-1}{1-\alpha}+C |\widetilde{w}|.  
\end{equation}
Combining \eqref{dist_lower_3} and \eqref{dist_lower_4}, we get 
\begin{equation}\label{dist_lower_5}
d_g(\widetilde{z}, \widetilde{w}) \geq  C_1 \Bigl(\frac{(1+\log (1+|\widetilde{z}|^2))^{1-\alpha}-1}{1-\alpha}+ \min\bigl(|\widetilde{w}|^{\frac{2N+2}{2N+3+\delta}}, |\widetilde{w}|\bigl)\Bigr).  
\end{equation}

\begin{comment}
 
\end{comment}
\vskip 0.1cm
\noindent \textbf{Step 2: Proof of the estimates on $|w|$ and $|\nabla w|^2$ in \eqref{sharp_property_est}.}

By \eqref{P_dis_est}, we see immediately that $|w| \leq C (1+d_g(z, w))^{\frac{2N+3+\delta}{2N+2}}$. The exponent $\frac{2N+3+\delta}{2N+2}$ in this estimate can not be lowered along the points $(0, w)$ as $|w| \rightarrow +\infty$.
As 
\[
|\nabla w|^2=\frac{1}{B}=s^{1+\delta} (1+\frac{c |w|^2}{N s^2})^{-N} \leq s^{1+\delta} \leq [C_1+C_2 d_g(z, w)]^{1+\delta}.
\] Here we use \eqref{s_R_relation} or \eqref{P_dis_est}. Note that the exponent $1+\delta$ in this estimate can not be lowered along the points $(z, 0)$ as $|z| \rightarrow +\infty$.

\vskip 0.1cm

\noindent \textbf{Step 3: Proof of the torsion estimate in \eqref{sharp_property_est}.}

We estimate the norm of torsion by \eqref{torsion_norm_B} with our choice of $P$ in \eqref{P_def_final}.

\begin{equation}\label{torsion_norm_B_est}
\begin{split}
    |T|^2=&\frac{t}{1+t} \Bigl[ L^{2\alpha} (1+\frac{2\alpha}{L})^2+(1-\alpha)^2 (P_s-\frac{1}{(1-\alpha)s})^2 \Bigr]\\
    =& \frac{t}{1+t} \Bigl[ L^{2\alpha} +o(L^{2\alpha})+\frac{1}{s^2}\bigl(2N(1-\alpha) \frac{\frac{c|w|^2}{Ns^2}}{1+\frac{c|w|^2}{Ns^2}}+1 \bigr)^2  \Bigr] \\
    \sim & (1+\log(1+|z|^2))^{2\alpha}+ o(L^{2\alpha}).
\end{split}    
\end{equation} By \eqref{P_dis_est}, we get that $|T| \leq C[1+d_g(z, w)]^{\frac{\alpha}{1-\alpha}}$. The exponent $\delta=\frac{\alpha}{1-\alpha}$ is also optimal along $(z, 0)$ as $|z| \rightarrow +\infty$.

\vskip 0.1cm
\noindent \textbf{Step 4: Proof of $Ric^{(2)}>0$ in \eqref{sharp_property_est}.}

It suffices to check \eqref{ricci_posi_condi} holds with our choice of $P$ in \eqref{P_def_final}.

Let $\xi=\frac{c|w|^2}{Ns^2}$. We solve
\[
P_s=-\frac{2N}{s}\frac{\xi}{1+\xi}, \ \ P_{ss}=\frac{2N\xi (3+\xi)}{s^2(1+\xi)^2},\ \ \ P_{rr}+\frac{P_r}{r}=\frac{4c}{s^2(1+\xi)^2}.
\]
Note that the term $t+2-\frac{\alpha t}{L}$ appearing in both formulas in \eqref{c(t)_plus_term} satisfies
\[
t+2-\frac{\alpha t}{L}>t+2-\alpha t> (1-\alpha)(t+2).
\]
Then $I(t, |w|)$ defined in \eqref{ricci_posi_condi} can be simplified 
\begin{equation}
  \begin{split}
      I(t, |w|)=&\frac{L^{\alpha} s(t+2-\frac{\alpha t}{L})-(1-\alpha) t}{(1+t) s^2}-\frac{(1-\alpha)^2 t}{1+t} \frac{2N \xi (3+\xi)}{s^2(1+\xi)^2}\\
      &+\frac{(1-\alpha)L^{\alpha}}{1+t}(t+2-\frac{\alpha t}{L}) \frac{2N}{s}\frac{\xi}{1+\xi}-\frac{s^{1+\delta}}{4(1+\xi)^N}\frac{4c}{s^2(1+\xi)^2}\\
      \geq & \frac{L^{\alpha} s(1-\alpha) (t+2)-(1-\alpha) t}{(1+t) s^2}-\frac{cs^{\delta}}{s(1+\xi)^{N+2}}\\
      &+ \frac{2N\xi(1-\alpha)^2}{s^2(1+\xi)^2(1+t)} \Bigl[ L^{\alpha} (t+2)s(1+\xi)-t(3+\xi)  \Bigr] \\
      \geq & \frac{\frac{5}{6} L^{\alpha} (1-\alpha)(t+2)}{(1+t)s}-\frac{c s^{\delta}}{s} \\
      &+ \frac{2N\xi(1-\alpha)^2}{s^2(1+\xi)^2(1+t)} \Bigl[  6(t+2)(1+\xi)-t(3+\xi)  \Bigr],\ \ \text{where we use}\ s \geq 6, L \geq 1\\
      > & \frac{\frac{5}{6} L^{\alpha} (1-\alpha)-c \, 6^{\frac{\alpha}{1-\alpha}} L^{\alpha}}{s},\ \ \text{since}\ s \leq 6\,L^{1-\alpha}. 
  \end{split}  
\end{equation}
To sum up, $Ric^{(2)}>0$ as long as $0<c<\frac{5(1-\alpha)}{6^{\frac{1}{1-\alpha}}}$. Hence the construction of Example \ref{sharp_G_intro} is complete.

\subsection{Examples of complete Chern-flat manifolds}

Recall that $\mathcal{O}(M)$ denotes the ring of holomorphic functions on a complex manifold $M$. Given $f_1,\dots,f_m\in\mathcal{O}(M)$, we denote by $\mathbb{C}\{f_1,\ldots,f_m\}$ the subring consisting of holomorphic functions representable by power series $\sum_{\nu \in(\mathbb{N} \cup \{0\})^m} a_\nu f^\nu$ where $a_\nu\in\mathbb{C}$ and $f^\nu=f_1^{\nu_1}\cdots f_m^{\nu_m}$ that converge normally on $M$. By normal convergence we mean that 
\[
\sum_{\nu  \in  (\mathbb{N} \cup \{0\})^m}  |a_\nu|\sup_{x\in K}|f^\nu(x)|<\infty
\]
for every compact subset $K\subset M$. For example, if $z_1,\dots,z_n$ are the standard global holomorphic coordinates on $\mathbb{C}^n$, then $\mathcal{O}(\mathbb{C}^n)=\mathbb{C}\{z_1,\dots,z_n\}$.

Recall that we define $f \in \mathcal{O}(M)$ has polynomial growth of order at most $d$ in \eqref{porder_atmost}.
Similarly, we say that $f \in \mathcal{O}(M)$ has \emph{Hadamard order $\leq \beta$} for some $\beta>0$ if there exists some $C(\beta, f)>0$ such that 
\begin{equation}\label{Horder_atmost}
\log |f(q)| \leq C(d(q, p)+1)^{\beta}, \ \ \ \ \forall q \in M.    
\end{equation}
For simplicity, we also define that $f$ has \emph{Hadamard order $\beta$} if its Hadamard order is at most $\beta$ and not at most $\beta-\delta$ for any $0<\delta < \beta$.

We study the function theory on complete Chern-flat manifolds in low dimensions.

\subsubsection{Dimension two}

Let $G$ be a simply connected complex Lie group in dimension $2$. Its Lie algebra $\mathfrak{g}$ is isomorphic to
\[
\mathfrak{g}=\operatorname{span}\{X, Y\},\ \ \ [X, Y]=\lambda Y,  \ \ \text{for some}\ \lambda \in \mathbb{C}.
\]
If $\lambda=0$, $G$ is isomorphic to the vector group $(\mathbb{C}^2, +)$. For any $\lambda \in \mathbb{C} \setminus \{0\}$, the corresponding group $G$ is isomorphic to any other such group. It is a solvable Lie group which is isomorphic to a semi-direct product $\mathbb{C} \ltimes \mathbb{C}$. The multiplication in $G$ can be expressed as
\begin{equation}
(t, s) \cdot (t^{\prime}, s^{\prime})=(t+t^{\prime}, s+e^{\lambda t}s^{\prime}).\ \ \ \label{2dim_group}    
\end{equation}
We may pick a left-invariant frame as $\{e_1=\frac{\partial}{\partial t},\ e_2=e^{\lambda t}\frac{\partial}{\partial s}\}$. Then the corresponding left-invariant Hermitian metric is
\begin{equation}
    \omega=\sqrt{-1}(dt \wedge d\overline{t}+e^{-2\operatorname{Re}(\lambda t)} ds \wedge d\overline{s}).  \label{2d_metric}
\end{equation}

The following result is a special case of Theorem \ref{func_Lie_summary}. We include a direct proof, since $(G,g)$ given by \eqref{2d_metric} is of independent interest.

\begin{proposition}\label{model_2dim}
Let $(M, g)$ be the Hermitian manifold $(\mathbb{C} \ltimes \mathbb{C}, \omega)$ defined in \eqref{2dim_group} and \eqref{2d_metric}. We assume $\lambda \neq 0$. Then for any $d \geq 0$, $\mathcal{O}_d(M, g)=\operatorname{span}\{1, t, t^2, \cdots, t^{\lfloor d \rfloor}\}$ where $\lfloor d \rfloor$ denotes the greatest integer not exceeding $d$.
\end{proposition}

\begin{proof}[Proof of Proposition \ref{model_2dim}]
   It suffices to consider the case of $\lambda=-1$. Let $t=x_1+\sqrt{-1}y_1$ and $s=x_2+\sqrt{-1}y_2$. Then $(M, g)$ is isometric to the product of $3$-dimensional hyperbolic metric $g_h$ and a real line $(\mathbb{R}, dy_1^2)$. Indeed, with a coordinate change $v=e^{-x_1}$, we have
   \[
   g_{h}= \frac{dv^2+dx_2^2+dy_2^2}{v^2},
   \] which is the standard hyperbolic metric on the upper half space $\mathbb{R}^3_{+}=\{(x_2, y_2, v) \ |\ v>0\}$. It is a standard fact that the geodesic sphere centered at $\widetilde{p}=(0, 0, 1)$ with radius $R$ can be solved as
   \begin{equation}
    x_2^2+y_2^2+(v-\cosh R)^2=\sinh^2 R.  \label{geo_shape}    
   \end{equation}
   It follows that $e^{-R} \leq v \leq e^R$ and $\sqrt{x_2^2+y_2^2} \leq \sinh R$. 
   
   Consider the point $p=(\widetilde{p}, 0) \in \mathbb{R}^3_{+} \times \mathbb{R}$ which corresponds to $(t, s)=(0, 0) \in M$. For any $q=(\widetilde{q}, y_1) \in M$ where $x_1=-\log v$, we have 
   \begin{equation}
    (d_g(q, p))^2=(d_{g_h}(\widetilde{q}, \widetilde{p}))^2+y_1^2.  \label{Pytho}
   \end{equation}
   Now we obtain the growth estimates of holomorphic functions $t$ and $s$.
   \begin{align}
   |t|&=\sqrt{(-\log v)^2+y_1^2}  \leq \sqrt{R^2+y_1^2} = d_g(q, p),   \label{tgrowth}\\
   |s|&=\sqrt{x_2^2+y_2^2} \leq \sinh R \leq e^{R} \leq e^{d_g(q, p)}.  \label{sgrowth}
   \end{align}
Let $B(p, R)$ denote the geodesic ball of $(M, g)$ centered at $p$ with radius $R$. Given any $f(t, s) \in \mathcal{O}_d(M)$, we show that $f$ is a polynomial of $t$ of degree at most $\lfloor d \rfloor$. To this end, we check that  for $R>0$ sufficiently large, the bi-circle $\{|t|=\frac{R}{2}, |s|=\frac{1}{2}e^{\frac{R}{4}}\}$ is contained in $B(p, 2R)$. In view of (\ref{geo_shape}) and (\ref{Pytho}), it suffices to check that
\begin{align*}
x_2^2+y_2^2+(v-\cosh R)^2 \leq \frac{1}{4}e^{\frac{R}{2}} +[\sinh R-\frac{1}{2}(e^{-\frac{R}{2}}-e^{-R})]^2<\sinh^2 R.   
\end{align*}
Then we obtain the following Cauchy estimates on a bidisk.
\begin{align}
|\frac{\partial^{k_1+k_2} }{\partial t^{k_1} \partial s^{k_2}} f(0, 0)| & \leq   \frac{ k_1 ! k_2 ! \sup_{B(p, 2R)} |f(t, s)|}  {(\frac{R}{2})^{k_1} (\frac{1}{2}e^{\frac{R}{4}})^{k_2}} \leq \frac{ k_1 ! k_2 ! C(1+2R)^d}  {(\frac{R}{2})^{k_1} (\frac{1}{2}e^{\frac{R}{4}})^{k_2}}.  \label{cauchy1}
\end{align}
Therefore, for any $k_1>d$ or $k_2 \geq 1$, the corresponding partial derivatives $\frac{\partial^{k_1+k_2} }{\partial t^{k_1} \partial s^{k_2}} f(0, 0)=0$. That is to say, any $f \in \mathcal{O}_d(M, g)$ is a polynomial of $t$ with degree $ \leq \lfloor d \rfloor$.

\begin{comment}
 Moreover, there exists some $R_0>0$ so that for the coordinate slice $\{|t|=R, s=0\}$ is contained in $\overline{B(p, 2R)}$ any $R>R_0$. The latter statement follows once we show 
\begin{equation}
\Phi(x_1, y_1)=(e^{x_1}-\cosh(2R))^2+y_1^2-\sinh^2(2R) < 0\ \text{when}\ |t|=\sqrt{x_1^2+y_1^2}=R>R_0.
\label{Property1}
\end{equation}
Once (\ref{Property1}) is proved, 

It remains to verify (\ref{Property1}). To that end, we rewrite
\[
\Phi(x)=e^{2x}-2\cosh (2R)e^x-x^2+(R^2+1),\ \ x \in [-R, R].
\] Then $F(-R)=-e^R (1-O(e^{-R}))$ and $F(R)=-e^{2R}(1-O(e^{-R}))$. And it is elementary to show
$F'<0$ on $[-R, R]$ if $R$ is large enough.    
\end{comment}

\end{proof}

Next, we discuss a generalization of the complete Chern-flat Hermitian metrics constructed in \cite[p.~5767]{WYZ2020}. Given any holomorphic function $\rho(z)$ on $\mathbb{C}$, we define a Hermitian metric $g$ on ${\mathbb C}^2$. 
\begin{equation}
   \omega= \sqrt{-1} ( \varphi^1\wedge \overline{\varphi^1} + \varphi^2 \wedge \overline{\varphi^2}), 
   \label{C2Herm}
\end{equation}
where $\varphi^1=dz_1$, $\varphi^2=e^{\rho(z_1)} dz_2$, and $(z_1, z_2)$ is the standard coordinate of $\mathbb C^2$. The torsion components under the corresponding $(1, 0)$ frame 
are $T^1_{12}=0$, $T^2_{12}=\frac{1}{2} \rho'(z_1)$. So the norm of the Chern torsion $|T|^2=|\rho'(z_1)|^2$ is not a constant unless $\rho(z)$ is linear. If $\rho(z)=-\lambda z$ for some constant $\lambda \in \mathbb{C}$, then $(\mathbb{C}^2, g)$ is holomorphically isometric to $\mathbb{C} \ltimes \mathbb{C}$ given by (\ref{2dim_group}) and (\ref{2d_metric}).

\begin{proposition}\label{2dim_more}

For any holomorphic function $\rho$ on $\mathbb{C}$, \eqref{C2Herm} defines a complete Chern-flat Hermitian metric on $\mathbb{C}^2$. Assume that $\rho(z)$ is a polynomial of degree $k$ with $k \geq 1$. Then for any $d \geq 0$, $\mathcal{O}_d(M, g)=\operatorname{span}\{1, z_1, z_1^2, \cdots, z_1^{\lfloor d \rfloor}\}$. Moreover, $z_2$ is of Hadamard order $k$.
\end{proposition}

\begin{proof}[Proof of Proposition \ref{2dim_more}]

\textbf{Step 1:} We show that $g$ defined in (\ref{C2Herm}) is complete. Let $\gamma : [0,\infty ) \rightarrow {\mathbb C}^2$ be a smooth curve that goes to infinity. Write $\gamma (t) = (z_1(t), z_2(t))$.  Its length with respect to $g$ is
\begin{equation}
L_g(\gamma)= \int_0^{\infty } \sqrt{ |z_1'(t)|^2 + e^{2\operatorname{Re}(\rho(z_1))}|z_2'(t)|^2 } dt.  \label{lengthf}    
\end{equation}
It suffices to show that $L_g(\gamma)$ is infinite. Assume the contrary. Then $\int |z_1'|dt \le L_g(\gamma) < \infty$. Thus $|z_1(t)|\leq C_1$ for some constant $C_1$ for any $t>0$. Let $C_2=\inf_{|z_1| \leq C_1} \operatorname{Re}\rho(z_1)$. Then $L_g(\gamma) \geq e^{C_2} \int |z_2'(t)|dt$. It follows that $z_2(t)$ is also bounded. This contradiction shows that $g$ is complete.

\textbf{Step 2:} We study the growth of holomorphic functions with respect to (\ref{C2Herm}) if $\rho(z)$ is a polynomial of degree $k \geq 1$. For simplicity, we focus on the case $\rho(z)=a_k z^k$ with $a_k \neq 0$. The general case can be treated by a similar argument.

Let $O$ be the origin of $\mathbb{C}^2$. It follows from (\ref{lengthf}) that $|z_1(q)| \leq d_g(q, O)$ and the equality is attained at any point $(z_1, 0)$. Hence $z_1$ grows linearly. Assume that $\gamma(t)$ with $t_0 \leq t \leq t_1$ is the minimizing geodesic from $O$ to any given point $q$. By (\ref{lengthf}), we have
\[
L(t) \coloneqq d(O, \gamma(t))=\int_{t_0}^t \sqrt{|z_1^{\prime}(t)|^2+e^{2\operatorname{Re}(a_k z_1^k(t))}|z_2^{\prime}(t)|^2} dt.  
\]
Now we observe the following inequality
\[
L'(t) \geq e^{-|a_k|[L(t)]^k}|z_2^{\prime}(t)|\ \  \text{for}\ t_0 \leq t \leq t_1.
\]
We choose $t_{\ast} \in (0, t_1]$ so that $L(t_{\ast})=1$. It follows from an integration from $t_{\ast}$ to $t_1$ that
\begin{equation}
|z_2(q)| \leq C(1+d_g(q, O))(e^{|a_k| (d_g(q, O))^k}+1)      \label{z2exp_est}
\end{equation}
where $C$ is a constant which depends on $\sup_{B_g(O, 1)} |z_2|$. Therefore $z_2$ is of Hadamard's order $\leq k$.

Pick any $\delta>0$ which is sufficiently small. Assume that there exist some $C>0$ such that 
\begin{equation}
  |z_2(q)| \leq \exp{(C(d(q, O)+1)^{k-2\delta})}  \ \ \text{for any }\ q \in \mathbb{C}^2.   \label{contragrow}
\end{equation}
Then we consider another curve $\gamma(t)=(z_1(t), z_2(t))=(t e^{i\theta_0}, \exp(t^{d_1}))$ with $t \in [0, \infty)$ and $d_1=k-\delta$. Here we choose $\theta_0 \in [0, 2\pi)$ so that $\operatorname{Re} \rho(z_1(t))=-|a_k|t^k$. Then
\begin{align}
d_g(\gamma(t_n), \gamma(0)) &\leq \int_0^{t_n} \sqrt{|z_1'(t)|^2+e^{2\operatorname{Re}\rho(z_1(t))}|z_2'(t)|^2} dt    \nonumber \\
&\leq t_n+ \int_0^{+\infty}   e^{-|a_k|t^k+t^{d_1}} (d_1 t^{d_1-1}) dt.  \label{z2exp_est2}
\end{align}
In view of (\ref{contragrow}), we get a contradiction from the above inequality as $t_n \rightarrow \infty$.

\textbf{Step 3:} For any $d \geq 0$, we determine $\mathcal{O}_d(M, g)=\operatorname{span}\{1, z_1, z_1^2, \cdots, z_1^{\lfloor d \rfloor}\}$.

Let $\rho(z)=a_kz^k+\cdots+a_1 z+a_0$ with $a_k \neq 0$. Fix any integer $d_1>d$. As in the proof of Proposition \ref{model_2dim}, we show that when $R$ is sufficiently large, the coordinate circle 
$\{|z_1|=\frac{R}{4}, |z_2|=\bigl(\frac{R}{4}\bigr)^{d_1}\}$ is contained in $B(O, 2R)$. Now we pick any point $q$ with $|z_1(q)|=\frac{R}{4}, |z_2(q)|=\bigl(\frac{R}{4}\bigr)^{d_1}$. Assume that $z_2(q)=\bigl(\frac{R}{4}\bigr)^{d_1} e^{i\theta_1}$ for some $\theta_1 \in [0, 2\pi)$. We consider a piecewise smooth curve from $O$ to $q$ so that
\[
\gamma_1(t)=(t e^{i\theta_0}, t^{d_1}e^{i\theta_1})\ \ \ 0 \leq t \leq \frac{R}{4},  \ \ \ \gamma_2(s)=((1-s)\frac{R}{4}e^{i\theta_0}+sz_1(q), z_2(q))\ \ 0 \leq s \leq t_2.
\] Here we pick $\theta_0$ such that there exist $c_1$ and $c_2$ such that $\operatorname{Re}(\rho(te^{i\theta_0})) \leq c_1-c_2\,t^k$ for any $t \geq 0$. We may estimate the lengths of $\gamma_1$ and $\gamma_2$ by (\ref{lengthf}). It follows from the triangle inequality that
\begin{equation}
d_g(O, q) \leq  L(\gamma_1(t))+L(\gamma_2(s))  \leq \frac{R}{4}+e^{c_1} \int_0^{\infty} e^{-c_2t^k} d_1 t^{d_1-1}dt+\frac{R}{2},  \label{twolength}
\end{equation}
Let $c_3=e^{c_1} \int_0^{\infty} e^{-c_2t^k} d_1 t^{d_1-1}dt$. For any $R>c_3$, it follows from (\ref{twolength}) that $q \in B(p, 2R)$. The desired conclusion on $\mathcal{O}_d(M, g)$ follows from the Cauchy estimates as in (\ref{cauchy1}).

\end{proof}

\begin{remark}\label{rho(z)general}
Let $\{\widetilde{e}_i\}_{i=1}^4$ be an orthonormal basis corresponding to $\{e_1, e_2\}$ in Proposition \ref{2dim_more}.
\[
e_1=\frac{1}{\sqrt{2}}(\widetilde{e}_1-\sqrt{-1}\widetilde{e}_2),\ \  e_2=\frac{1}{\sqrt{2}}(\widetilde{e}_3-\sqrt{-1}\widetilde{e}_4). \]
Indeed, we solve its Riemannian curvature by \cite[Lemma 7]{YZ2018}.
\begin{align*}
&Ric(\widetilde{e}_3, \widetilde{e}_3)=-|\rho'(z_1)|^2=2R(\widetilde{e}_3, \widetilde{e}_4, \widetilde{e}_4, \widetilde{e}_3),\\
&Ric(\widetilde{e}_1, \widetilde{e}_1)=-\frac{1}{2}|\rho'(z_1)|^2-\operatorname{Re}\Big(\frac{1}{2}(\rho'(z_1))^2+\rho^{\prime\prime}(z_1)\Big).    
\end{align*}
If $\rho(z_1)$ is a polynomial of degree $\geq 2$, the Riemannian Ricci curvature is not bounded from below. Moreover, it has mixed signs if $\rho(z_1)$ is suitably chosen. Compare the case where $\rho(z)$ is linear and the corresponding Hermitian structure is $H^3 \times \mathbb{R}$, as in Proposition \ref{model_2dim}. 
\end{remark}

\begin{comment}
    Let $\widetilde{z}_1=e^{i\theta} z_1$. It follows from (\ref{C2Herm}) that
\begin{equation}
   \widetilde{\omega}= \sqrt{-1} ( d \widetilde{z}_1 \wedge  d\overline{\widetilde{z}_1} + e^{2\operatorname{Re}\rho(e^{-i\theta}\widetilde{z}_1)} dz_2 \wedge d\overline{z_2} ) 
   \label{C2Herm_2}
\end{equation}
 So the growth orders of $\widetilde{z}_1$ and $z_2$ with respect to $\widetilde{\omega}$ are the same as those of $z_1$ and $z_2$ with respect to $\omega$. 
\end{comment}

\subsubsection{Dimension three: the solvable case}\label{3dim_examp}

It is well-known that three-dimensional solvable complex Lie algebra has been classified, see \cite{PZ1990} and \cite{degraaf2005}. The corresponding simply connected complex Lie groups are listed as follows. 

\begin{enumerate}
    \item  The abelian case: $G=(\mathbb{C}^3, +)$.
    \item  The nilpotent case: $G$ is the complex Heisenberg Lie group with its Lie algebra $\mathfrak{g}$ having basis $e_1, e_2, e_3$ satisfying $[e_1, e_2]=e_3$ and $[e_1, e_3]=[e_2, e_3]=0$.
    \item  A non-nilpotent family: $G$ is the semi-direct product $\mathbb{C} \ltimes \mathbb{C}^2$. The corresponding Lie algebra is spanned by $T, Z, W$ with
\[
[T, Z]=Z, \ \ [T, W]=\lambda W,\ \ \ [Z, W]=0, \ \ \text{for some}\ \lambda \in \mathbb{C}. 
\] Note that $\operatorname{ad}_T$ acts on $\operatorname{span}\{Z, W\}$ diagonally.
The corresponding group law is
\[
(t, z, w) \cdot (t^{\prime}, z^{\prime}, w^{\prime})=(t+t^{\prime}, z+e^{t}z^{\prime}, w+e^{\lambda t}w^{\prime}).
\]
\item  Another non-nilpotent group: $G$ is the semi-direct product of $\mathbb{C} \ltimes \mathbb{C}^2$. The corresponding Lie algebra is spanned by $T, Z, W$ such that
\[
[T, Z]=Z, \ \ [T, W]=Z+ W,\ \ \ [Z, W]=0. 
\] Note that $\operatorname{ad}_T$ acts on $\operatorname{span}\{Z, W\}$ non-diagonally. The corresponding group law is
\[
(t, z, w) \cdot (t^{\prime}, z^{\prime}, w^{\prime})=(t+t^{\prime}, z+e^{t}(z^{\prime}+t\,w^{\prime}), w+e^{t}w^{\prime}).
\]
\end{enumerate}

In this subsection, we show that the complex Heisenberg group admits many complete Chern-flat Hermitian metrics that are not left-invariant; see Proposition \ref{Heisen_non_inv}.

Recall that the complex Heisenberg Lie group $G$ consists of all complex $3\times 3$ matrices $X$ in the form
\begin{equation} \label{Heisen_def}
X = \left( \begin{array}{lll} 1 & z_1 & z_3 \\ 0 & 1 & z_2 \\ 0 & 0 & 1
\end{array} \right).     
\end{equation} 
The left-invariant holomorphic $1$-forms $\varphi^1=dz_1$, $\varphi^2=dz_2$, and
$\varphi^3=dz_3-z_1dz_2$ define a natural Hermitian metric with zero Chern curvature
\begin{equation}
\omega=\sqrt{-1} (\sum_{k=1}^3 \varphi^k \wedge \overline{\varphi^k}).   \label{3dim_metric1}
\end{equation} Moreover, it is balanced as $\eta=0$ in (\ref{1formdef}). 

As a complex manifold, $G$ is biholomorphic to $\mathbb{C}^3$ and $\mathcal{O}(G)=\mathbb{C}\{z_1,z_2,z_3\}$. The Riemannian geometry of (\ref{3dim_metric1}) has been studied extensively. We refer to \cite{Eber1994} and \cite{Pansu1983} for results on curvature, geodesics, and volume growth of geodesic balls for the left-invariant geometry of nilpotent groups. For example, its Riemannian sectional curvature is bounded below by a negative constant, and the volume of a geodesic ball $B(p, r)$ grows like $r^8$ as $r$ tends to infinity.

Growth of holomorphic functions on $G$ can be studied using Theorem \ref{func_Lie_summary}.  Note that the global holomorphic coordinates $z_1, z_2$, and $z_3$ given by \eqref{Heisen_def} do not coincide with exponential coordinates in Theorem \ref{func_Lie_summary}. Nonetheless, it is easy to show that both $z_1$ and $z_2$ have linear growth, and $z_3$ has quadratic growth with respect to \eqref{3dim_metric1}. In fact, we may construct a family of complete Chern-flat Hermitian metrics on $G$.

\begin{comment}
\begin{proof}[Proof of Proposition \ref{3dimH_f}]
   Consider any $C^1$ curve $\gamma(t)=(z_1(t), z_2(t), z_3(t))$ from the identity $p=(0, 0, 0)$ to a point 
   $q=(z_1, z_2, z_3)$. We have $L_g(\gamma(t)) \geq \sup_{\gamma(t)} |z_1(t)|$. In the meantime, we estimate
   \begin{align*}
       |z_3| &\leq \int |z_3'(t)| dt \leq \int |z_3'(t)-z_1(t)z_2'(t)| dt +\int |z_1(t)z_2'(t)| dt\\
       &\leq \int |\gamma'(t)|_g dt + [L_g(\gamma(t))]^2.
   \end{align*}
   It follows that $|z_3|(q)  \leq d_g(q, p)+[d_g(q, p)]^2 \leq C(1+d_g(q, p))^2$ for some constant $C>0$. Moreover, the quadratic order is sharp along the curve $\gamma(t)=(t, 2t, t^2)$ for $t \in [0, \infty)$ as
   \[
   \frac{|z_3|}{(d_g(\gamma(t), p)+1)^2} \geq \frac{t^2}{(\sqrt{5}t+1)^2}.
   \]
\end{proof}    
\end{comment}

\begin{proposition}\label{Heisen_non_inv}
    For any two holomorphic functions $\rho$ and $\xi$ on $\mathbb{C}^2$, we introduce three global holomorphic $1$-forms on $G$ as follows
    \[
    \psi^1=dz_1,\ \ \  \psi^2=dz_2, \ \ \ \psi^3=dz_3-\rho(z_1, z_2)dz_1-\xi(z_1, z_2)dz_2.
    \]
    Then the Hermitian metric 
    \begin{equation}
      \omega=\sqrt{-1} (\sum_{k=1}^3 \psi^k \wedge \overline{\psi^k})   \label{3dim_metric2}  
    \end{equation}
    is a complete Chern-flat metric. It is not holomorphically isometric to any left-invariant metric as long as $\frac{\partial \rho}{\partial z_2}-\frac{\partial \xi}{\partial z_1}$ is non-constant. Note that $z_1$ and $z_2$ are of linear growth. Moreover, $z_3$ is of polynomial growth if both $\rho$ and $\xi$ are polynomials.    
\end{proposition}

\subsubsection{Dimension three: the semisimple case}

By the Levi decomposition, the only (simply connected) semisimple complex Lie group in dimension $3$ is $\operatorname{SL}(2, \mathbb{C})$. 

Now we review a natural left-invariant Hermitian metric on $\operatorname{SL}(2, \mathbb{C})$. We begin with the following choice of a basis in $\mathfrak{g}=\mathfrak{sl}(2,\mathbb{C})$ as
\[
X=\begin{pmatrix}
    1 & 0 \\
    0  & -1
\end{pmatrix},\ \ Y=\begin{pmatrix}
    0 &  1 \\
    0  & 0
\end{pmatrix},\ \ Z=\begin{pmatrix}
    0 & 0 \\
    1  & 0
\end{pmatrix}, 
\] with the relation
\[
[X, Y]=2Y,\ \ \ [X, Z]=-2Z,\ \ \ [Y, Z]=X.
\]
Assume $G=\operatorname{SL}(2, \mathbb{C})$ as follows
\begin{equation}\label{SL2C_def}
\operatorname{SL}(2, \mathbb{C})=\Bigl\{ \begin{pmatrix}
    z_1 & z_2 \\
    z_3  & z_4
\end{pmatrix} \ \ |\ \  z_1z_4-z_2z_3=1, \ z_i \in \mathbb{C},\ \ 1 \leq i \leq 4 \Bigr\}.     
\end{equation}

Now we solve a left-invariant frame on $G$.
\begin{align}
e_1&=\left.\frac{d}{dt}\right|_{t=0} \Big[\begin{pmatrix}
    z_1 & z_2 \\
    z_3  & z_4
\end{pmatrix} \exp(t X)\Big]=z_1 \frac{\partial}{\partial z_1}-z_2 \frac{\partial}{\partial z_2}+z_3 \frac{\partial}{\partial z_3}-z_4 \frac{\partial}{\partial z_4}.\\
e_2&=z_1 \frac{\partial}{\partial z_2}+z_3 \frac{\partial}{\partial z_4},\ \ \   e_3=z_2 \frac{\partial}{\partial z_1}+z_4 \frac{\partial}{\partial z_3}. \nonumber
\end{align}
The corresponding dual frame is
\begin{align}
\varphi^1=z_4 dz_1-z_2dz_3, \ \ \varphi^2=z_4 dz_2-z_2 dz_4,\ \ \varphi^3=z_1 dz_3-z_3dz_1.
\end{align}
Therefore, we may define a left-invariant Hermitian metric on $\operatorname{SL}(2, \mathbb{C})$
\begin{equation}
\omega=\sqrt{-1} (\sum_{i=1}^3  \varphi^i \wedge \overline{\varphi^i}).  \label{ss_metric}    
\end{equation}
Note that
\begin{align*}
d\varphi^1=-\varphi^2 \wedge \varphi^3, \ \ d\varphi^2=-2\varphi^1 \wedge \varphi^2,\ \ d\varphi^3=2\varphi^1 \wedge \varphi^3.
\end{align*}
It follows that $T_{23}^1=-\frac{1}{2}, T_{12}^2=-1$, and $T_{13}^3=1$. As a result, $\eta$ in (\ref{1formdef}) vanishes and $\omega$ defined by \eqref{ss_metric} is balanced. 

We have seen that both the complex Heisenberg group and $\operatorname{SL}(2, \mathbb{C})$ are balanced with respect to their left-invariant geometry. Indeed, any left-invariant Hermitian metric on a complex Lie group $G$ is balanced if and only if $\operatorname{tr} \operatorname{ad}_X =0$ for all $X\in \mathfrak{g}$. So we have the following observation.

\begin{proposition}\label{Lie_balanced}
   Let $G$ be a complex Lie group. If its Lie algebra $\mathfrak{g}$ is nilpotent or perfect, then any left-invariant Hermitian metric on $G$ is balanced.
\end{proposition}

To see that Proposition \ref{Lie_balanced} holds, we recall the Jacobi identity: $\operatorname{ad}_{[X,Y]}=[\operatorname{ad}_X, \operatorname{ad}_Y]$, $\forall X, Y \in \mathfrak{g}$. Taking the trace on both sides, we have
$\operatorname{tr} ad_{[X,Y]} = \operatorname{tr} ([\operatorname{ad}_X, \operatorname{ad}_Y])=0$. Then if $\mathfrak{g}$ is perfect ($\mathfrak{g}=[\mathfrak{g},\mathfrak{g}]$) we have $\operatorname{tr} ad_X =0$ for all $X\in \mathfrak{g}$. In particular, Proposition \ref{Lie_balanced} holds for any semisimple Lie group.

As a complex manifold, we have $\mathcal{O}(G)=\mathbb{C}\{z_1, z_2, z_3, z_4\}/(z_1z_4-z_2z_3-1)$. As a result of Theorem \ref{poly_space_char_intro}, there is no nonconstant holomorphic function of polynomial growth on $G$. Can we classify all complete Chern-flat Hermitian metrics on $\operatorname{SL}(2, \mathbb{C})$? To the best of our knowledge, Question \ref{ques3_intro} remains unanswered. It is interesting to note that $\operatorname{SL}(2, \mathbb{C})$ admits a complete asymptotically conical Calabi-Yau K\"ahler metric by \cite{Stenzel1993}.

\subsubsection{A non-K\"ahler quotient space obtained from $\operatorname{SL}(2, \mathbb{C})$}\label{quotientSL}

Usually we prefer the right coset space as it is compatible with the left-invariant geometry. However, in the following Proposition \ref{basicfunc} and Theorem \ref{non_kahler}, we use the left coset space for the setting of principal bundles.

\begin{proposition}\label{basicfunc}
Let $\Gamma_{\mathbb{Z}}=\begin{pmatrix}
  1 &  \mathbb{Z} \\
  0 &  1
\end{pmatrix}$ and $\Gamma_{\mathbb{C}}=\begin{pmatrix}
  1 &  \mathbb{C} \\
  0 &  1
\end{pmatrix}$. 
Then $\operatorname{SL}(2, \mathbb{C})/\Gamma_{\mathbb{Z}}$ is not holomorphically convex.  
\end{proposition}

\begin{proof}[Proof of Proposition \ref{basicfunc} using a result in \cite{BO1973}]
Let $\begin{pmatrix}
    a &  b  \\  c  &  d  
\end{pmatrix} \in \operatorname{SL}(2, \mathbb{C})$. By Barth--Otte \cite{BO1973}, any right $\Gamma_{\mathbb{Z}}$-invariant holomorphic function on $\operatorname{SL}(2, \mathbb{C})$ is also right $\Gamma_{\mathbb{C}}$-invariant. Obviously, the function $\begin{pmatrix}
    a &  b  \\  c  &  d  
\end{pmatrix} \rightarrow a$ and $\begin{pmatrix}
    a &  b  \\  c  &  d  
\end{pmatrix} \rightarrow c$ are right $\Gamma_{\mathbb{C}}$-invariant. Consider any holomorphic function $F$ on $\operatorname{SL}(2, \mathbb{C})$ which is right $\Gamma_{\mathbb{C}}$-invariant. We introduce
\begin{align*}
    &\widetilde{f}(a, c) \coloneqq F(\begin{pmatrix}
    a &  0  \\  c  &  \frac{1}{a}  
\end{pmatrix}),\ \ \ \ (a, c) \in \mathbb{C}^2 \setminus \{a=0\}. \\
    &\widehat{f}(a, c) \coloneqq F(\begin{pmatrix}
    a &  -\frac{1}{c}  \\  c  &  0  
\end{pmatrix}),\ \ \ \ (a, c) \in \mathbb{C}^2 \setminus \{c=0\}. 
\end{align*}
Note that $\widetilde{f}$ and $\widehat{f}$ coincide on the overlap as $\begin{pmatrix}
    a &  0  \\  c  &  \frac{1}{a}  
\end{pmatrix} \begin{pmatrix}
    1 &  -\frac{1}{ac}  \\  0  &  1  
\end{pmatrix}  =\begin{pmatrix}
    a &  -\frac{1}{c}  \\  c  &  0  
\end{pmatrix} $. So they together define a holomorphic function $f$ on $\mathbb{C}^2 \setminus (0, 0)$, which extends to $\mathbb{C}^2$ by Hartogs' theorem. To sum up, we have shown that any right $\Gamma_{\mathbb{C}}$-invariant holomorphic function $F$ can be identified with a holomorphic function $f$ of $a$ and $c$. So the ring of right $\Gamma_{\mathbb{Z}}$-invariant holomorphic functions on $\operatorname{SL}(2, \mathbb{C})$ is $\mathbb{C}\{a, c\}$. However, any such holomorphic function remains constant along a sequence $\{p_n\}_{n=1}^{\infty} \in \operatorname{SL}(2, \mathbb{C})$ where $p_n=\begin{pmatrix}
    1 &  n\,i  \\  0  &  1  
\end{pmatrix}$. We note that $[p_n]$ is a divergence sequence in $\operatorname{SL}(2, \mathbb{C})/\Gamma_{\mathbb{Z}}$.
\end{proof}

Now we review an interesting example where the total space of a principal bundle over K\"ahler manifold is not K\"ahler. This example is due to Berteloot--Oeljeklaus {\cite[Theorem~3.1]{BO1988}} and Greb--Miebach {\cite[Example~5.1]{GM2023}}. We consider the right action of $\Gamma_{\mathbb{C}}$ on $\operatorname{SL}(2, \mathbb{C})$. Note that the projection to the first column defined by $\begin{pmatrix}
    a &  b  \\  c  &  d  
\end{pmatrix}  \rightarrow   (a, c)
$ is a biholomorphic map between $\operatorname{SL}(2, \mathbb{C})/\Gamma_{\mathbb{C}}$ and $\mathbb{C}^2 \setminus \{0\}$.
We may check that $\operatorname{SL}(2, \mathbb{C})/ \Gamma_{\mathbb{Z}} \rightarrow  \operatorname{SL}(2, \mathbb{C})/ \Gamma_{\mathbb{C}}$ defines a $\mathbb{C}^{\ast}$-principal bundle over $\mathbb{C}^2 \setminus \{0\}$.

\begin{theorem}[{\cite[Theorem 3.1]{BO1988}}]\label{non_kahler}
$\operatorname{SL}(2, \mathbb{C})/\Gamma_{\mathbb{Z}}$ admits no K\"ahler metrics.
\end{theorem}

In fact, a more general result for complex semisimple Lie groups was proved in \cite[Theorem~3.1]{BO1988}. The proof relies on a related result due to Huckleberry; see \cite[Theorem~3.2]{BO1988}. We briefly recall the main idea of the argument as follows. If $\operatorname{SL}(2, \mathbb{C})/\Gamma_{\mathbb{Z}}$ admits a K\"ahler form $\omega$, then $\operatorname{SL}(2, \mathbb{C})$ admits a right $\Gamma_{\mathbb{Z}}$-invariant K\"ahler form $\omega$. Then we could find a smooth strictly PSH function $\Psi$ on $\operatorname{SL}(2, \mathbb{C})$ such that $\omega=\sqrt{-1}\partial\overline{\partial} \Psi$. We proceed to show that a suitable modification of $\Psi$, denoted by $\widetilde{\Psi}$, also defines a smooth K\"ahler metric on $\operatorname{SL}(2, \mathbb{C})$. Moreover, $\widetilde{\Psi}$ is right $\Gamma_{\mathbb{Z}}$-invariant. According to a result of Berteloot \cite{Bert1987}, it is also right $\Gamma_{\mathbb{C}}$-invariant. However, $\Gamma_{\mathbb{C}}$ defines an entire curve in $\operatorname{SL}(2, \mathbb{C})$ along which $\sqrt{-1}\partial\overline{\partial} \widetilde{\Psi}$ vanishes. We refer the reader to \cite{BO1988} and \cite{Bert1987} for the detailed proof.

In view of Theorem \ref{non_kahler}, we conclude the paper with the following remark.

\begin{remark}\label{quotient_final}
Consider $M=\Gamma_{\mathbb Z}\backslash \operatorname{SL}(2,\mathbb C)$ equipped with the natural Chern-flat Hermitian metric descending from \eqref{SL2C_def} and \eqref{ss_metric}. We observe that $\mathcal{O}(M)=\mathbb{C}\{z_3,z_4\}$. Moreover, by Theorem \ref{poly_space_char_intro}, there are no nonconstant holomorphic functions of polynomial growth on $M$. Does $M$ admit another complete Hermitian metric which is better adapted to its function theory? 
\end{remark}

\bibliographystyle{amsplain}

\bibliography{holo_related}

\end{document}